\documentclass{amsart}
\usepackage{amssymb,epsfig, amsmath, amsfonts}

\usepackage{ifpdf}
\usepackage{color}
\usepackage{bm}
\usepackage{dsfont} 
\usepackage[labelfont=rm,labelformat=simple]{subcaption}
  	\DeclareCaptionSubType*{figure}
	
\usepackage{enumitem}
\usepackage{upgreek}
\usepackage{mathrsfs}
\usepackage[plain]{algorithm}
\usepackage{algorithmic}
\usepackage{framed}
\usepackage{afterpage}
\usepackage{graphicx}
\usepackage{xspace}
\usepackage{todonotes}
\usepackage{xparse}
\usepackage{url}
\usepackage{paralist}
\usepackage{wrapfig}

\usepackage{siunitx}

\newtheorem{theorem}{Theorem}[section]
\newtheorem{lemma}[theorem]{Lemma}
\newtheorem{proposition}[theorem]{Proposition}
\newtheorem{corollary}[theorem]{Corollary}

\theoremstyle{definition}
\newtheorem{definition}[theorem]{Definition}

\theoremstyle{remark}
\newtheorem{remark}[theorem]{Remark}

\numberwithin{equation}{section}

\floatstyle{ruled}
\newfloat{algo}{!ht}{loa}
\floatname{algo}{Algorithm}

\newcommand{\STOPCRIT}[1]{%
    \STATE\algorithmicif\ {#1}\ \algorithmicthen\ \algorithmicreturn\ %
}

\usepackage{tikz}
\usepackage{pgfplots}
\usepackage{pgfplotstable}
\usetikzlibrary{calc}
\usetikzlibrary{positioning, fit}
\usetikzlibrary{decorations.pathreplacing}
\usetikzlibrary{shapes.symbols}
\usetikzlibrary{backgrounds}

\DeclareFontFamily{U}{mathx}{\hyphenchar\font45}
\DeclareFontShape{U}{mathx}{m}{n}{
     <5> <6> <7> <8> <9> <10>
      <10.95> <12> <14.4> <17.28> <20.74> <24.88>
      mathx10
      }{}
\DeclareSymbolFont{mathx}{U}{mathx}{m}{n}
\DeclareFontSubstitution{U}{mathx}{m}{n}
\DeclareMathAccent{\mywidehat}{0}{mathx}{"70} 
\DeclareMathAccent{\mywidecheck}{0}{mathx}{"71}

\newcommand{\ind}[1]{\mathds{1}_{\{#1\}}}

\DeclareMathOperator{\supp}{supp}
\DeclareMathOperator{\Span}{span}

\DeclareMathOperator{\argmax}{argmax}

\makeatletter
\newcommand\etal{et~al\@ifnextchar.{}{.\@\xspace}}
\newcommand\eg{e.g\@ifnextchar.{}{.\@\xspace}}
\newcommand\ie{i.e\@ifnextchar.{}{.\@\xspace}}
\newcommand\confer{cf\@ifnextchar.{}{.\@\xspace}}
\newcommand\etc{etc\@ifnextchar.{}{.\@\xspace}}
\makeatother

\newcommand{\cAmtree}{\mathfrak{A}_{\mathrm{mtree}}}

\newcommand{\per}{\mathrm{per}}

\newcommand{\E}{\mathbb{E}}
\newcommand{\R}{\mathbb{R}}
\newcommand{\N}{\mathbb{N}}

\newcommand{\cA}{\mathcal{A}}
\newcommand{\cB}{\mathcal{B}}

\newcommand{\cD}{\mathcal{D}}

\newcommand{\cJ}{\mathcal{J}}

\newcommand{\cL}{\mathcal{L}}
\newcommand{\cM}{\mathcal{M}}
\newcommand{\cN}{\mathcal{N}}

\newcommand{\cO}{\mathcal{O}}

\newcommand{\cR}{\mathcal{R}}

\newcommand{\cX}{\mathcal{X}}
\newcommand{\cY}{\mathcal{Y}}
\newcommand{\cZ}{\mathcal{Z}}

\newcommand{\fA}{\mathfrak{A}}

\newcommand{\bb}{\mathbf{b}}

\newcommand{\bbf}{\mathbf{f}}
\newcommand{\bg}{\mathbf{g}}
\newcommand{\br}{\mathbf{r}}
\newcommand{\bu}{\mathbf{u}}
\newcommand{\bv}{\mathbf{v}}
\newcommand{\bw}{\mathbf{w}}
\newcommand{\bx}{\mathbf{x}}

\newcommand{\bA}{\mathbf{A}}
\newcommand{\bB}{\mathbf{B}}

\newcommand{\bE}{\mathbf{E}}

\newcommand{\bcJ}{\bm{\cJ}}
\newcommand{\blambda}{\boldsymbol \lambda}
\newcommand{\bLambda}{\boldsymbol \Lambda}

\newcommand{\bpsi}{\boldsymbol \psi}
\newcommand{\bPsi}{\boldsymbol \Psi}

\newcommand{\bXi}{\boldsymbol \Xi}

\newcommand{\hatbcJ}{\hspace{0.9ex}\mywidehat{\hspace{-0.9ex}\bm{\cJ}}}
\newcommand{\checkbcJ}{\hspace{0.9ex}\mywidecheck{\hspace{-0.9ex}\bm{\cJ}}}

\newcommand{\checkbpsi}{\mywidecheck{\bpsi}}

\newcommand{\hatbpsi}{\mywidehat{\bpsi}}

\newcommand{\checkbPsi}{\mywidecheck{\bPsi}}
\newcommand{\hatbPsi}{\mywidehat{\bPsi}}

\newcommand{\checkbLambda}{\mywidecheck{\bLambda}}
\newcommand{\hatbLambda}{\mywidehat{\bLambda}}

\newcommand{\checkbXi}{\mywidecheck{\bXi}}

\newcommand{\eps}{\varepsilon}

\newcommand{\eval}[2]{\langle #1 , #2 \rangle}

\newcommand{\norm}[1]{\lVert #1\rVert}

\NewDocumentCommand{\errestappr}{ m m o}{%
	\IfValueTF{#3}
	{\Delta_{#1,#3}(#2)}
	{\Delta_{#1}(#2)}
}

\allowdisplaybreaks[4]

\title[Adaptive Reduced Basis Methods]{Reduced Basis Methods With Adaptive Snapshot Computations}
\thanks{This work has partly been supported by the Deutsche Forschungsgemeinschaft (DFG) within the Research Training Group (Graduiertenkolleg) GrK1100 \emph{Modellierung, Analyse und Simulation in der Wirtschaftsmathematik} at Ulm University.}
\date{\today}

\author{Mazen Ali}
\address{
Mazen Ali,
Institute for Numerical Mathematics,
University of Ulm,
Helm\-holtz\-strasse 20,
D-89069 Ulm, Germany}
\email{mazen.ali@uni-ulm.de}

\author{Kristina Steih}
\address{
Kristina Steih,
Institute for Numerical Mathematics,
University of Ulm,
Helm\-holtz\-strasse 20,
D-89069 Ulm, Germany}
\email{kristina.steih@uni-ulm.de}

\author{Karsten Urban}
\address{
Karsten Urban,
Institute for Numerical Mathematics,
University of Ulm,
Helm\-holtz\-strasse 20,
D-89069 Ulm, Germany}
\email{karsten.urban@uni-ulm.de}

\makeatletter
\@namedef{subjclassname@2010}{%
  \textup{2010} Mathematics Subject Classification}
\makeatother



\begin{document}
\setdefaultleftmargin{1em}{2em}{}{}{}{}

\begin{abstract}
We use asymptotically optimal \emph{adaptive} numerical methods (here specifically a wavelet scheme) for  snapshot computations within the offline phase of the Reduced Basis Method (RBM). The resulting discretizations for each snapshot (i.e., parameter-dependent) do not permit the standard RB `truth space', but allow for error estimation of the RB approximation with respect to the exact solution of the considered parameterized partial differential equation. 	

The residual-based a posteriori error estimators are computed by an adaptive dual wavelet expansion, which allows us to compute a surrogate of the dual norm of the residual. The resulting adaptive RBM is analyzed. We show the convergence of the resulting adaptive Greedy method. Numerical experiments for stationary and instationary problems underline the potential of this approach.
\end{abstract}

\maketitle

\section{Introduction}
Reduced Basis Methods (RBMs) have nowadays become a widely accepted and used tool for realtime and/or multi-query simulations of parameterized partial differential equations (PPDEs). By using an offline-online decomposition, the main idea is to use a high fidelity, detailed, but costly numerical solver offline to compute approximations to the PPDEs for certain parameter values. The selection of these parameters is done by an error estimator which is efficiently computable and thus allows one to determine the `worst' parameters out of a possibly rich so-called training set. For those `bad' parameters, the high fidelity model is used in order to determine approximations, so-called \emph{snapshots}. These few snapshots form the reduced basis which is then capable to produce approximations for any new parameter value extremely rapidly (online). The error estimator can also be used online in order to certify this RB approximation. Both the variety of applications and the amount of recent results in RBMs go well beyond the scope of {this} introduction. 

The success of this `classical' RBM also relies on the assumption that the high fidelity model in the offline phase is sufficiently accurate \emph{for all} parameters. The same discretization is used for all snapshots. This may have some possible drawbacks: (1) If this high fidelity model is not accurate enough, also the RB-approximation cannot be good. (2) The other extreme is that a sufficiently accurate approximation for all possible parameters may require a high fidelity model whose dimension is too large even for an offline phase. (3) The error estimate usually controls the difference to the high fidelity solution, not w.r.t.\ the exact solution of the PPDE (with one recent exception in \cite{Masa} to be discussed below).

On the other hand, there are adaptive numerical methods available that guarantee an approximation of the exact solution of a PDE within a preselected tolerance. Such methods can, e.g., be based upon finite element or wavelet discretizations, \cite{CDD01,CDD02,NSV,Urban:WaveletBook}. We use such an adaptive method (we choose wavelets) for computing snapshots in the offline phase. This offers some features that we think are of interest, namely: (a) We use different discretizations for each parameter allowing for a (asymptotically) minimal amount of work for any chosen parameter. (b) We can bound the RB error w.r.t.\ the exact solution of the PPDE. (c) We introduce a new surrogate for the infinite-dimensional exact residual as well as its dual norm by using a dual wavelet expansion. The resulting error estimator is shown to be online-efficient.
	
Using adaptivity (or different discretizations) in the offline phase implies some additional sophistication of the method, at least from the conceptual point of view. The question arises under which circumstances such adaptivity might pay off. It is known, e.g., from \cite{CDD01} that adaptive methods show faster convergence rates if the Besov regularity of the solution in a certain scale exceeds the Sobolev regularity, see also \cite{Urban:WaveletBook}. For the offline RB setting this means that the regularity of the solution \emph{with respect to the parameter} is of crucial importance. If one single discretization is sufficient for approximating the solution $u(\mu)$ well enough for all possible parameters $\mu$, then adaptivity does not to make sense. On the other hand, if $u(\mu)$ significantly differs w.r.t.~$\mu$, a joint discretization may be too fine. This is, e.g., the case if $u(\mu)$ has {strong} parameter-dependent local effects. Our numerical examples are guided by these considerations.

The adaptive offline snapshot computation gives rise to some implications that we discuss. Once having an adaptive method at hand, the question of convergence and a posteriori error analysis arises. Even though the dual norm of the residual is a rigorous error bound, its computation would require to solve an infinite-dimensional problem. We introduce a surrogate by using the expansion of the residual in terms of the dual wavelet basis which in turn admits a characterization of the dual norm.

We would like to mention that there is existing literature for RBM and various flavors of adaptivity, e.g., sampling set randomization, adaptive refinement of training sets, hp-RBM, time-partitioning etc., see, e.g., \cite{MR3043486,MR2452388,Carlberg14,MR2882313,MR3177844,MR3123826}, just to mention a few.

The remainder of this paper is organized as follows. In Section \ref{sec:rbm}, we review the main facts of the `classical' Reduced Basis Method. We set the framework for PPDEs and collect those facts that are needed here. Section \ref{sec:adaptiverbm} is devoted to the use of adaptive methods for the generation of the reduced basis in the offline phase. At this point, we only require the availability of a certain adaptive solver \textbf{SOLVE} and do not specify which specific method is used. We have used an Adaptive Wavelet Galerkin Method (AWGM) which is briefly described  in Section \ref{sec:awgm}. Within the adaptive RB framework in Section \ref{sec:adaptiverbm}, however, it is not  necessary to fix the precise adaptive method. There, we just assume that a surrogate for the residual-based error estimator is computable. Such a surrogate is described in Section \ref{sec:awgm} using the dual wavelet system. In Section \ref{sec:numerics}, we describe numerical experiments for two different examples, namely heat conduction in a thermal block with several local heat sources and time-dependent convection-diffusion-reaction using a space-time variational formulation. These experiments do not only confirm theoretical findings quantitatively but also indicate the potential of the new approach.

\section{Reduced Basis Methods (RBMs)}\label{sec:rbm}
In order to highlight differences and challenges of using adaptively computed basis functions within the Reduced Basis Method (RBM), it makes sense to briefly review `standard' RBMs.

\subsection{Parameterized Partial Differential Equations (PPDEs)}\label{Sec:2.1}
Let $\Omega\subset\R^n$ be a bounded domain on which we consider function spaces $\cX = \cX(\Omega)$, $\cY = \cY(\Omega)$ arising from a variational formulation of a partial differential equation. Denoting by $\cD \subset \R^{P}$ the set of parameters, this means that we consider a differential operator $\cB: \cD \times \cX \to \cY'$ resp.\ a bounded bilinear form $b:\cX\times\cY\times\cD \to \R$, where $b(w,v;\mu) := \eval{\cB(\mu)w}{v}_{\cY'\times\cY}$ for $w\in\cX$, $v\in\cY$ and $\mu\in\cD$. In particular, we assume the existence of constants $\gamma(\mu)\le\gamma^\text{UB}<\infty$ such that
\begin{equation}\label{Eq:bBound}
	b(w,v; \mu) \le \gamma(\mu)\, \| w\|_\cX\, \| v\|_\cY, \qquad w\in\cX, v\in\cY.
\end{equation}
For a given $f(\mu)\in\cY'$, the problem is then to find a $u(\mu) \in \cX$ such that $\cB(\mu)\, u(\mu) = f(\mu)$  in  $\cY'$, or, in variational form
\begin{equation} \label{eq:variationalproblem}
 	b(u(\mu), v;\mu) = f(v;\mu) \qquad \forall\, v \in \cY,
\end{equation}
where $f(v;\mu) := \eval{f(\mu)}{v}_{\cY'\times\cY}$.

We assume that \eqref{eq:variationalproblem} is well-posed for all $\mu\in\cD$, which is equivalent to the so-called \emph{Ne\v{c}as condition} on $b(\cdot,\cdot; \mu)$, \cite{Necas,NSV}, \ie, there exist \emph{inf-sup constants} $\beta(\mu)$ and a lower bound $\beta_\text{LB}$ such that
\begin{equation}\label{Necas} 
	\beta(\mu)  
	= \inf_{w \in \cX}\sup_{v \in \cY} \frac{b(w,v;\mu)}{\norm{w}_{\cX}\norm{v}_{\cY}} 
	= \inf_{w \in \cY}\sup_{u \in \cX} \frac{b(w,v;\mu)}{\norm{w}_{\cX}\norm{v}_{\cY}} 
	\geq \beta_\text{LB} > 0 
\end{equation}
for all $\mu \in \cD$.

\begin{remark} 
(a) It is worth mentioning that \eqref{eq:variationalproblem} includes elliptic problems, where, e.g., $\cX = \cY = H^1_0(\Omega)$ (or other boundary conditions), $b(\cdot,\cdot;\mu)$ being coercive with constant $\alpha(\mu) > 0$, as well as parabolic initial value problems in space-time formulation, \ie, with the Bochner spaces $\cX = W_{0}(0,T;V) := \{u \in L_{2}(0,T;V) : u_{t} \in L_{2}(0,T;V'), u(0)=0 \in H\}$, $\cY = L_{2}(0,T;V)$, $V=H^1_0(\Omega)$, so that $\cX \subsetneqq \cY$, and also time-periodic problems, see also Section \ref{sec:numerics} below. \\
(b) Instead of a space-time formulation for a parabolic initial value problem, one could also use a standard time-stepping scheme. There are corresponding RBMs available for such problems \cite{GreplThesisPaper, PODGreedy}. In principle, our subsequent findings can be extended also to those settings, but in order to keep notations simple, we restrict ourselves to \eqref{eq:variationalproblem}.  
\end{remark}

For later reference, we consider the \emph{residual}, which is defined for any $w\in\cX$ by
\begin{equation}\label{Def:Res}
	r_b(w;\mu):= f(\mu)-\cB(\mu)w \in\cY', 
\end{equation}
i.e., $\langle r_b(w;\mu),v \rangle_{\cY'\times\cY} := f(v;\mu) - b(w, v;\mu)$, $v\in\cY$. It is then straightforward and well-known that
\begin{equation}\label{Eq:ResEq}
	\beta(\mu)\, \| u(\mu)-w\|_\cX  
	\le  \| r_b(w;\mu) \|_{\cY'} 
	\le \gamma(\mu)\, \| u(\mu)-w\|_\cX.
\end{equation}

\subsection{Some Basics on `Classical' RBMs}
Any numerical scheme for the solution of \eqref{eq:variationalproblem} involves a discretization of $\cX$, $\cY$. In a standard RB setting these finite-dimensional discrete spaces, the so-called \emph{truth spaces}, are denoted by $\cX^{\cN}\subset \cX$, $\cY^{\cN} \subset \cY$.\footnote{We always use calligraphic symbols for high-(even $\infty$)-dimensional spaces.}
Then, the following Petrov-Galerkin projection is considered:
\begin{equation}\label{eq:truthproblem}
 	\text{Find }u^{\cN}(\mu) \in \cX^{\cN}:\qquad b(u^{\cN}(\mu), v;\mu) = f(v;\mu) \qquad \forall\, v \in \cY^{\cN},
\end{equation}
where inf-sup-stability is assumed, i.e.,
\begin{equation}\label{Eq:LBB-cN}
	\beta^\cN(\mu) := 
	\inf_{w^\cN\in X^\cN} \sup_{v^\cN\in Y^\cN} \frac{ b(w^\cN, v^\cN; \mu)}{\| w^\cN\|_\cX\, \| v^\cN\|_\cY} 
	\ge \tilde\beta >0,
\end{equation}
with $\tilde\beta$ independent of $\cN$ as $\cN\to\infty$. 
Often, $\cX^{\cN}$, $\cY^{\cN}$ are spanned by \emph{local basis functions} such as finite elements or wavelets and their dimension $\cN = \dim(\cX^\cN)=\dim(\cY^\cN)$\footnote{For simplicity, we assume that trial and test spaces are of the same dimension. Otherwise, one would need to use a least squares approach.} is usually large, so that solving \eqref{eq:truthproblem} repeatedly for many different parameters would be too costly or realtime computations would be impossible.

\begin{remark}[Fixed discretization]
We stress that in the standard RB setting, the spaces $\cX^{\cN}$, $\cY^{\cN}$ are a-priorily \emph{fixed} and are the \emph{same} for all parameters $\mu \in \cD$. Moreover, it is assumed that the discretization error $\norm{u(\mu) - u^{\cN}(\mu)}_\cX$ is negligibly small for all $\mu \in \cD$. Thus, typical RBMs view $u^{\cN}(\mu)$ as `truth', which means, \eg, that all error estimates are typically w.r.t.\ $u^{\cN}(\mu)$ and do \emph{not} take $u(\mu)$ into account. Just recently a first paper appeared introducing error bounds w.r.t.\ $u(\mu)$ in a specific case {of symmetric coercive problems, \cite{Masa}. To the best of our knowledge, the techniques in \cite{Masa} are at least not immediately applicable to non-symmetric Petrov-Galerkin-type problems \eqref{eq:variationalproblem} using adaptive discretizations.} 
\end{remark}

The idea behind (standard) RBMs is the construction of low-dimensional spaces $X_{N}^{\cN} \subset \cX^{\cN}$, $Y_{N}^{\cN} \subset \cY^{\cN}$ (which may also be parameter-dependent, i.e., $Y_{N}^{\cN}(\mu)$, see \cite{MR2281777} and also our construction below), $N \ll \cN$,\footnote{Low-dimensional spaces are denoted by usual (non calligraphic) symbols.} from so-called \emph{snapshots}, \ie, solutions of \eqref{eq:truthproblem} for selected parameters $\mu^1, \ldots, \mu^N$, \ie,
\begin{equation}\label{XNN}
	X_{N}^{\cN} := \Span\{u^{\cN}(\mu^{i}), i=1,\dots,N\} =: \Span\{\zeta_{i}^{\cN},i=1,\dots,N\},
\end{equation}
and $Y_{N}^{\cN} := \Span\{\eta_{i}^{\cN},i=1,\dots,N\}$ is such that the $N$-dimensional reduced problem
\begin{equation}\label{eq:RB-infsup}
	\text{Find }u_{N}^{\cN}(\mu) \in X_{N}^{\cN}:\qquad b(u_{N}^{\cN}(\mu), v;\mu) = f(v;\mu) \qquad \forall\, v \in Y_{N}^{\cN}
\end{equation}
is stable as $N\to\infty$. Stability in the discrete setting is ensured by the fulfillment of a \emph{discrete inf-sup condition} \cite{Babuska}, i.e.,
\begin{equation}\label{Eq:LBB-N}
	\beta^\cN_N(\mu) := 
	\inf_{w_N\in X_N^\cN} \sup_{v_N\in Y_N^\cN} \frac{ b(w_N, v_N; \mu)}{\| w_N\|_\cX\, \| v_N\|_\cY} 
	\ge \bar\beta >0,
\end{equation}
with $\bar\beta$ independent of $N$ as $N\to\infty$.  
We abbreviate
\begin{equation}\label{Def:SN}
	S_N := \{ \mu^1, \ldots , \mu^N\}
\end{equation}
as the set of (sample) parameter values corresponding to the snapshots. The system $\{\zeta_{i}^{\cN},i=1,\dots,N\}$ may arise by orthonormalization of the snapshots.

The inf-sup condition \eqref{Necas} gives rise to rigorous a posteriori \emph{error bounds}, \ie,  quantities $\Delta_{N}^{\cN}(\mu)$ with
\begin{equation}\label{eq:trutherrorbound}
	\norm{e^{\cN}_{N}(\mu)}_{\cX} := \norm{u^{\cN}(\mu) - u^{\cN}_{N}(\mu)}_{\cX} \leq \Delta_{N}^{\cN}(\mu) 
		= \frac{\norm{r_{b}^{\cN}(u_N(\mu);\mu)}_{\cY'}}{\beta{^\cN}(\mu)},
\end{equation}
where $r_{b}^{\cN}(u_N(\mu);\mu):\cY^{\cN}\to\R$ is the \emph{`truth' residual} with respect to the reduced solution, \ie,
\[
	\langle r_{b}^{\cN}(u_N(\mu);\mu), v\rangle_{\cY'\rtimes\cY} := f(v;\mu) - b(u_{N}(\mu),v;\mu), \qquad \forall\,v \in \cY^\cN.
	\] 
Note, that $\Delta_{N}^{\cN}(\mu)$ can be computed online efficient, i.e., with cost independent of $\cN$. In \eqref{eq:trutherrorbound}, $\beta^\cN(\mu)$ denotes the inf-sup constant of $b(\cdot,\cdot;\mu)$ w.r.t.\ the truth spaces, \ie, \eqref{Eq:LBB-N} with $X_N^\cN$, $Y_N^\cN$ replaced by $\cX^{\cN}$, $\cY^{\cN}$, respectively. We call this RB-standard estimator \emph{residual-based.}
The involved dual norms $\norm{r_{b}^{\cN}(u_N(\mu);\mu)}_{\cY'}$ are computed with the help of the \emph{Riesz representations}. 
Using inf-sup-stability and continuity yields -- similar to \eqref{Eq:ResEq} --
that the error estimator and the error are in fact equivalent{\footnote{{One can improve this estimate by using the continuity $\gamma^\cN(\mu)$ of $b(\cdot,\cdot;\mu)$ on $\cX^\cN$, $\cY^\cN$.}}}:
\begin{equation}\label{Eq:Equiv}
	 \| e_N^\cN(\mu) \|_\cX \le \Delta_N^\cN(\mu) \le \frac{\gamma(\mu)}{\beta^{{\cN}}(\mu)} \| e_N^\cN(\mu) \|_\cX.
\end{equation}

\begin{remark}\label{Rem:Repro}
	(a) We point out (for later reference in \S \ref{sec:Repro} below) that $u_N^\cN(\mu)=u^\cN(\mu)$ for all $\mu\in S_N$, i.e., snapshots are reproduced by the standard RBM. In fact, we have Petrov-Galerkin orthogonality, i.e., $b(u^\cN(\mu)-u_N^\cN(\mu), v_N; \mu)=0$ for all $v_N\in Y_N^\cN$. Since $u^\cN(\mu)\in X_N^\cN$ for $\mu\in S_N$, we have that $e^\cN_N(\mu) = u^\cN(\mu)-u_N^\cN(\mu)\in X_N^\cN$ and then \eqref{Eq:LBB-N} yields
	$
	\bar\beta \| e_N^\cN(\mu)\|_\cX \le \sup_{v_N\in Y_N^\cN} \frac{b(e_N^\cN(\mu), v_N; \mu)}{\| v_N\|_\cY} = 0,
	$
	i.e., $u^\cN(\mu)=u^\cN_N(\mu)$ for all $\mu\in S_N$.
	
	(b) The latter argument gives also rise to a straightforward estimate for the error w.r.t.\ the exact snapshot $u(\mu)$. In fact, for $\mu\in S_N$, triangle inequality yields $\| u(\mu)-u_N^\cN(\mu)\|_\cX \le \| u(\mu)-u^\cN(\mu)\|_\cX$, i.e., reproduction of the exact snapshot up to the tolerance of the truth approximation. Of course, this upper bound cannot be evaluated a posteriori in an efficient way. 
	\qed
\end{remark}

\subsection{Basis Construction via the Greedy Algorithm}

The choice of the RB basis functions $\zeta_{i}^{\cN}$, $i=1, \ldots ,N$, \ie, the selection of the corresponding parameter values $\mu^1, \ldots , \mu^N$, is often done using a Greedy algorithm: given $\mu^1, \ldots, \mu^n$, $n<N$, the next parameter value $\mu^{n+1}$ is chosen as
$$
	\mu^{n+1} = \arg \max_{\mu\in\cD_{\text{train}}} \Delta_{n}^{\cN}(\mu),
$$
where $\cD_{\text{train}} \subset \cD$ is a finite \emph{training set}. The key point for the efficiency of this approach is the fact that the greedy selection is done w.r.t.\ the error estimator {(which can be computed with cost independent of $\cN$)}. Only for the chosen parameter values $\mu^1, \ldots ,\mu^N$ the (expensive) truth has to be computed. The corresponding algorithm is displayed in Algorithm \ref{alg:TruthGreedy}. Note that this procedure is also called \emph{weak Greedy training}, in contrast to an (inefficient) \emph{strong} Greedy, where the true error $\| e_N^\cN(\mu)\|_\cX$ is used in line \ref{Alg:mustar}. More precisely, if $\| e_N^\cN(\mu)\|_\cX\le \gamma \Delta^\cN_N$ for some $\gamma>0$, then Algorithm \ref{alg:TruthGreedy} is called \emph{$\gamma$-weak Greedy}.

\begin{algo}
\caption{[$X_{N}^{\cN}$] = \textbf{Greedy}[tol, $N_{\max}$, $\cD_{\text{train}}$]}
\label{alg:TruthGreedy}
\begin{algorithmic}[1]
\STATE $S_0:=\emptyset$
\FOR{$N=1,\dots,N_{\max}$}
	\STATE Choose $\mu^{N} := \argmax_{\mu \in \cD_{\text{train}}}\Delta_{N-1}^{\cN}(\mu)$.\label{Alg:mustar}
	\STOPCRIT{$\Delta_{N-1}^{\cN}(\mu^{N}) < \text{tol}$}
	\STATE $S_N:=S_{N-1}\cup \{ \mu^N\}$.
	\STATE Compute snapshot $u^{\cN}(\mu^{N})$, update basis: $X_{N}^{\cN} = X_{N-1}^{\cN}{\oplus\Span}\{ u^{\cN}(\mu^{N})\}$.
	\STATE $N \gets N+1$.
\ENDFOR
\end{algorithmic}
\end{algo}
If Algorithm \ref{alg:TruthGreedy} stops with $N<N_{\max}$, then -- by \eqref{Eq:Equiv} -- we have
$$
	\max_{\mu\in \cD_{\text{train}}} \| e^\cN_N(\mu)\|_\cX
	\le \max_{\mu\in \cD_{\text{train}}} \Delta_{N}^{\cN}(\mu)
	< \text{tol},
$$
which means that the worst case error w.r.t.\ the parameter can be well-controlled provided that $\cD_{\text{train}}$ represents the full $\cD$ ``sufficiently well''.

\begin{remark}\label{Rem:Opt}
	{As an alternative to the Greedy algorithm one could determine $\mu^{N}$ by nonlinear optimization, \cite{MR2452388,Zeeb}.}
\end{remark}

\subsection{Offline-online Decomposition}\label{Sec:OO}

A crucial assumption for the efficiency of the RBM (in particular the efficient computation of $u^\cN_N(\mu)$ and of $\Delta_N^\cN(\mu)$) is that the bilinear form and the right-hand side are \emph{affine in the parameter}, i.e.,
\begin{equation}\label{eq:affinedecomp}
	b(u,v;\mu) = \sum_{q=1}^{Q_{b}} \theta_{b}^{(q)}\!(\mu)\, b^{(q)}(u,v), 
	\qquad 
	f(v;\mu) = \sum_{q=1}^{Q_{f}} \theta_{f}^{(q)}\!(\mu)\, f^{(q)}(v).
\end{equation}
Techniques like the Empirical Interpolation Method (EIM) \cite{Barrault2004667} can construct an approximation of such an affine decomposition if assumption \eqref{eq:affinedecomp} is not met. Affine forms as in \eqref{eq:affinedecomp} enable an efficient \emph{offline-online} decomposition of the calculations in the following sense: the parameter-independent components of the linear system, namely $\bB_N^{(q)}:= \left[b^{(q)}(\zeta_{i}^{\cN},\xi_{j}^{\cN})\right]_{i,j=1,\dots,N}$, $q=1,\dots,Q_{b}$, and $\bbf_N^{(q)} := \left[f^{(q)}(\xi_{j}^{\cN})\right]_{j=1,\dots,N}$, $q=1,\dots,Q_{f}$, can be precomputed (offline) so that the assembly and solution of the reduced system $\bB_N(\mu) \bu_N(\mu) = \bbf_N(\mu)$ with 
\begin{equation}\label{eq:LGSN}
	\bB_N(\mu) := \sum_{q=1}^{Q_{b}}\theta_{b}^{(q)}\!(\mu)\,\bB^{(q)},\qquad 
	\bbf_N(\mu) := \sum_{q=1}^{Q_{f}}\theta_{f}^{(q)}\!(\mu)\,\bbf^{(q)}, 
\end{equation}
for a new parameter $\mu\not\in S_N$ then only involves  $N$-dimensional matrix-vector products and can be done online (with complexity independent of $\cN$). Since $\bB_N(\mu)\in\R^{N\times N}$ is usually densely populated, the linear system for the reduced system to determine
$$
	u_N^\cN(\mu) = \sum_{i=1}^N u_N^{(i)}(\mu)\, \zeta_i^\cN,
	\qquad
	\bu_N(\mu) = [ u_N^{(i)}(\mu) ]_{i=1,\ldots ,N},
$$
can be solved with $\cO(N^3)$ operations -- independent of $\cN\gg N$. 
Also the error estimate can be computed online-efficient (independent of $\cN$).  
\section{Adaptive Reduced Basis Generation}\label{sec:adaptiverbm}
In this section, we describe those issues that arise when avoiding fixed truth spaces $\cX^\cN$ and $\cY^\cN$ and using adaptive methods based upon the infinite-dimensional spaces $\cX$ and $\cY$ instead. We assume that we have the following routine \textbf{SOLVE}  for the solution of a general operator equation $\cA x=b$ in $\cY'$, $\cA:\cX\to\cY'$ being a linear operator, at our disposal (not only for the specific operator $\cB(\mu)$ introduced above). The approximation produced by \textbf{SOLVE} will be measured in terms of  an error quantity $\E(x,\tilde x):\cX\times\cX\to\R^+$, $x, \tilde x\in\cX$, to be detailed later.

\smallskip
\noindent
\framebox{\begin{minipage}{0.97\textwidth}
	\textbf{SOLVE}: $[\cA, b, \eps]\mapsto x^\eps$: Approximation of $x:=\cA^{-1}b$ with $\E(x,x^\eps)\le \eps$ and in {optimal} complexity {in the sense of nonlinear approximation (see Theorem \ref{theo:quasioptimalAWGM})}.
	\end{minipage}}	
\smallskip

In Section \ref{sec:awgm}, we detail one possibility to realize \textbf{SOLVE} by an Adaptive Wavelet Galerkin Method (AWGM), but one could also use other schemes with the above properties such as adaptive finite element methods, see, e.g., \cite{NSV} for an overview.

\subsection{Adaptive Snapshot Computation}

With such an (adaptive) numerical solver \textbf{SOLVE} at our disposal, we  compute so-called \emph{$\eps$-exact residual approximations} $u^{\eps}(\mu)$ of $u(\mu) \in \cX$, such that
\begin{equation}\label{Eq:approximate}
	\E(u(\mu), u^{\eps}(\mu))  
	\leq \eps(\mu),
\end{equation}
where the approximation tolerance $\eps(\mu)$ depends on $\mu$ and will be specified later. This means in particular that there is no common `truth' space that all snapshots belong to -- but each approximation $u^\eps(\mu)$ belongs to a space $\cX_\mu^\eps$ that is determined adaptively (and whose dimension $\cN(\mu,\eps)$ is `large' from an RB point of view, but minimal in an adaptive approximation theory sense). 
The lack of common truth spaces for all parameters necessitates a re-interpreta\-tion of some RB ingredients which we will describe now. The reduced space is now spanned by \emph{approximate} snapshots computed during the offline training phase, \ie,  
\begin{equation}\label{eq:AdaptiveRBSpace}
 	X_{N}^\eps := 
	\Span\{\zeta_{i}^\eps,\, i=1,\dots,N\},
	\qquad
	\zeta_{i}^\eps := u^{\eps}(\mu^{i})\,\,\,
	\textrm{(or by orthogonalization)}
\end{equation}
and the reduced solution $u_{N}^\eps(\mu) \in X_{N}^\eps$ is the Petrov-Galerkin projection onto this space and the corresponding reduced inf-sup-stable, possibly parameter-dependent test space $Y_{N}^\eps(\mu)$ in the sense that
$$
	\inf_{w_N\in X^\eps_N} \sup_{v_N\in Y_N^\eps(\mu)} \frac{b(w_N, v_N; \mu)}{\| w_N\|_\cX\, \| v_N\|_\cY}
	\geq \beta_N(\mu) \geq \beta_{\text{LB}} > 0,\quad \forall\mu\in\cD.
$$
Note, however, that the value of $\beta_N(\mu)$ has to be expected to vary significantly with $\mu\in\cD$. 
The adaptive setting now also allows us to estimate the error with respect to the exact solution in $\cX$, \ie,
\begin{equation}\label{eq:exacterr}
	e_{N}^\eps(\mu) := u(\mu) - u_{N}^\eps(\mu),
\end{equation}
and not (only) the error w.r.t.\ a fixed and a priori given truth discretization. In fact, using standard arguments yields {a residual-based estimate analogous to \eqref{eq:trutherrorbound}}
\begin{equation}\label{Eq:EquivEps}
	\beta(\mu)\, \norm{e^{\eps}_{N}(\mu)}_{\cX} 
		\leq R_{b,N}^{\eps}(\mu) 
		\leq \gamma(\mu) \norm{e^{\eps}_{N}(\mu)}_{\cX},
\end{equation}
where $R_{b,N}^{\eps}(\mu):= \| r_{b}(u^\eps_N(\mu); \mu)\|_{\cY'}$ defined by \eqref{Def:Res}. This means that 
\begin{equation}\label{Eq:Delta-eps-N}
	\Delta_N^\eps(\mu) := \frac{R_{b,N}^{\eps}(\mu)}{\beta(\mu)\footnotemark}
\end{equation}
\footnotetext{{Note, that here in fact, we have $\beta(\mu)$ and not some reduced system analogue $\beta_N(\mu)$.}}%
is a surrogate for the true error $\| e_N^\eps(\mu)\|_\cX$. Note, however, that the computation of the residual  and its dual norm $R_{b,N}^{\eps}(\mu)$ requires the solution of an infinite-dimensional problem on $\cY$. 


\begin{remark}
	The above formulated adaptive framework can also be interpreted as using different finite element meshes for different  $\mu\in\cD$ in the snapshot generation.
\end{remark}

\subsection{Approximate Error Estimates}
\label{sec:equiverrorest}
In order to obtain a computationally feasible numerical method, we need a computable error estimator, recalling that in the infinite-dimensional setting neither the error $e^\eps_N(\mu)$ nor the (dual norm of the) residual $R_{b,N}^\eps(\mu)$ (and hence also the error estimator $\Delta_N^\eps(\mu)$ in \eqref{Eq:Delta-eps-N}) are computable. 

We shall assume that a \emph{surrogate} $\overline\Delta_{N}^\eps(\mu)$ is available (and computable) such that 
\begin{equation}\label{Eq:dRes}
	c_\Delta \, \Delta_N^\eps(\mu) 
	\leq 
	\overline\Delta_{N}^\eps(\mu) 
	\leq 
	C_\Delta\, \Delta_N^\eps(\mu).
\end{equation}
Of course, we have to expect that the complexity for the computation of $\overline\Delta_{N}^\eps(\mu)$ will grow as $c_\Delta, C_\Delta\to 1$. By \eqref{Eq:EquivEps}, this readily implies
\begin{equation}\label{Eq:dErr}
	\| e^\eps_N(\mu)\|_\cX 
	\le \frac1{c_\Delta} \overline\Delta_{N}^\eps(\mu) 
	\le \frac{C_\Delta}{c_\Delta} \frac{\gamma(\mu)}{\beta(\mu)}\, \| e^\eps_N(\mu)\|_\cX.
\end{equation}

\subsection{Adaptive Greedy Algorithm}
Now we have all ingredients at hand to formulate a fully adaptive version of the Greedy algorithm in Algorithm \ref{alg:EpsDeltaGreedy}. The adaptive computations take place in line \ref{eps:1} concerning the error estimator and in line \ref{eps:2} for the snapshot. 

\begin{algo}
	\caption{[$X_{N}^{\eps}$] = \textbf{AdaptGreedy}[$\widetilde{\text{tol}}$, $N_{\max}$, $\eps$, $\cD_{\text{train}}$]}
		\label{alg:EpsDeltaGreedy}
	\begin{algorithmic}[1]
		\STATE $S_0:=\emptyset$
		\FOR{$N=1,\dots,N_{\max}$}
			\STATE Choose $\mu^{N} := \argmax_{\mu \in \cD_{\text{train}}}\overline\Delta_{N-1}^{\eps}(\mu)$.\label{eps:1}
			\STOPCRIT{$\overline\Delta_{N-1}^{\eps}(\mu^{N}) < c_\Delta\, \widetilde{\text{tol}}$}\label{Alg:stop}
			\STATE $S_N:=S_{N-1}\cup \{ \mu^N\}$.
			\STATE Compute snapshot $u^{\eps}(\mu^{N})$ with \eqref{Eq:approximate}.\label{eps:2}
			\STATE Update reduced basis: $X_{N}^{\eps} = X_{N-1}^{\eps}{\oplus\Span}\{ u^{\eps}(\mu^{N})\}$.
			\STATE $N \gets N+1$.
		\ENDFOR
	\end{algorithmic}
\end{algo}


It seems natural that the tolerances $\eps(\mu)>0$ bound the reduction error $e^\eps_N(\mu) = u(\mu) - u^\eps_N(\mu)$ from below in the sense that the error cannot be smaller than the accuracy of the snapshot approximations. A result from \cite{GreedyConv} shows that this can lead to a stalling of the Greedy training at a certain level. As usual, the benchmark for the Greedy algorithm is the \emph{Kolmogorov $n$-width} for some $\Sigma\subset\cX$, i.e., 
$$
	d_n(\Sigma)  := \inf_{\dim(\cX_n)=n} \sup_{f\in\Sigma} \min_{g\in \cX_n} \| f-g\|_\cX.
$$

\begin{theorem}[\cite{GreedyConv}]\label{theo:GreedyConv}
Let $\cM(\mu):=\{ u(\mu):\, \mu\in\cD\}$ be compact and suppose that $d_{0}(\cM(\mu)) \leq M$, $d_{n}(\cM(\mu)) \leq M n^{-\theta}$ for some $M$, $\theta > 0$. Then, the approximation 
 \emph{$X^\eps_{N} = \mbox{\textbf{{AdaptGreedy}}[$\widetilde{\text{tol}}$, $N_{\max}$, $\eps$, $\cD$]}$} satisfies
\[
		\sup_{\sigma \in \cM(\mu)} \min_{g\in X^\eps_{N}} \| \sigma - g\|_\cX 
		\leq C(\theta, \varrho) \max\{M n^{-\theta}, \eps_{\text{UB}} \},\qquad
		\varrho := \frac{c_\Delta}{C_\Delta} \frac{\beta_{\text{LB}}}{\gamma_{\text{UB}}}
	\]
	with $\eps_{\text{UB}}:= \sup_{\mu\in\cD} \eps(\mu)$, $\beta_{\text{LB}}$ from \eqref{Necas} and $\gamma_{\text{UB}}$ from \eqref{Eq:bBound}.
	\hfill\qed
\end{theorem}

This result tells us that the RB Greedy training converges quasi-optimally compared to the Kolmogorov $n$-width until an $\eps$-dependent error level is reached. However, a closer look also shows why this result is not completely satisfying in the framework considered here:
\begin{compactitem}
	\item We face problems, where the dependency of all parameter-dependent quantities from the parameter is potentially strong (otherwise adaptivity is not justified). Thus, $\beta(\mu)$ and $\gamma(\mu)$ will strongly vary w.r.t.\ the choice of $\mu$. This, however, will cause the problem that the constant $\varrho$ is overly pessimistic.
	\item The appearance of $ \eps_{\text{UB}}$ seems to indicate that it does not pay off to compute snapshots with different accuracies, since at the end the poorest accuracy determines the overall quality. Again, for strong parameter influences, this is not appropriate, as we have also seen in various numerical experiments, e.g., \cite{Steih:Diss}.
\end{compactitem}
We conclude that a more refined error analysis is required.

\subsection{(Non-)Reproduction of Snapshots}\label{sec:Repro}
As we have pointed out in Remark \ref{Rem:Repro}, on a fixed truth discretization we have that $\Delta_{N}^{\cN}(\mu) = 0$ (up to numerical influences) for all $\mu \in S_{N}$, \ie, the error bound \emph{vanishes} on the set of snapshot parameters, as all snapshots $u^{\cN}(\mu) \in \cX^{\cN}$, $\mu\in S_N$, can be reconstructed \emph{exactly} from the basis functions and the Riesz representation for the error estimator is only based upon $\cX^\cN$, $\cY^\cN$.  
As we will explain now this is \emph{not} the case in the adaptive framework. The reason is that the approximate snapshot $u^\eps(\mu)$ is in $\cX_\mu^\eps$ but for the RB-approximation for the same parameter $\mu\in S_N$, we have that $u^\eps_N(\mu)\not\in\cX_\mu^\eps$. In fact, we only have 
\begin{equation}\label{Eq:CommonX}
	u_N^\eps(\mu) \in {\bigoplus_{\tilde\mu\in S_N}} \cX_{\tilde\mu}^\eps =: \cX^{\eps, S_N}.
\end{equation}
Hence, the argument using Petrov-Galerkin orthogonality as in Remark \ref{Rem:Repro} fails.
In fact, note that $e_N^\eps(\mu) = u(\mu)-u_N^\eps(\mu)$ is the error w.r.t.\ the unknown solution $u(\mu)$, whereas $e_N^\cN(\mu) = u^\cN(\mu)-u_N^\cN(\mu)$ involves the `truth' solution, which is in principle computable (up to numerical precision). This is important since in the `classical' case $u^\cN(\mu)$ is used  as a snapshot, whereas in the adaptive setting $u(\mu)$ cannot be computed and has to be replaced by an approximation $u^\eps(\mu)$. 
Hence, $b(u^\eps(\mu)-u_N^\eps(\mu), v_N^\eps; \mu)$ will in general \emph{not} vanish! This means that -- as opposed to the `classical' RBM -- snapshots are \emph{not reproduced} in the adaptive setting. Reproduction of RB basis functions is \emph{not} a consequence of the fact that the RB spaces are spanned by snapshots as RB basis functions, but a consequence of the Petrov-Galerkin orthogonality. 

Of course, one could use $\cX^{\eps, S_N}$ defined in \eqref{Eq:CommonX} as a joint common truth space as done, e.g., in \cite{Masa:New}. However, if the discretizations for various $\mu$ are significantly different, this would be by far too costly, in particular because already computed snapshots would have to be updated to the new truth space in each iteration. 
Hence, we face a reproduction error, which will be investigated below in more detail depending on the choice of the error measure $\E(\cdot,\cdot)$.

\subsection{Greedy convergence}
In \cite{Steih:Diss}, it was observed that snapshots might be multiply selected within the Greedy process. By suitably choosing the error measure $\E(\cdot,\cdot)$ and the snapshot tolerance $\eps(\mu)$ in \eqref{Eq:approximate}, we are now able to prove that the Greedy scheme with adaptive snapshot computation and an appropriate surrogate for the residual-based error estimator in Algorithm \ref{alg:EpsDeltaGreedy} in fact converges.

\begin{proposition}\label{prop:multiplesnapshots}
	Let  $\widetilde{\text{tol}}>0$ be a given Greedy tolerance. Moreover, we assume the following relation 
	\begin{equation}\label{eq:surrogate}
		\overline\Delta^\eps_n(\mu^i) \le C(\mu^i)\,\, \E(u(\mu^i), \zeta^\eps_i),
		\qquad
		\mu^i\in S_N,
		\quad \zeta^\eps_i:=u^\eps(\mu^i)
	\end{equation}
	for some $C(\mu^i)> 0$. Then, by setting in \eqref{Eq:approximate}
	\begin{equation}\label{Eq:choice-eps}
		\eps(\mu) := \widetilde{\text{tol}}\, \frac{c_\Delta}{C(\mu)}, 
	\end{equation}
	we have: if Algorithm \ref{alg:EpsDeltaGreedy} terminates for some $N<N_{\max}$, we get  
	\begin{equation}\label{Eq:final-error}
		\max_{\mu\in \cD_{\text{train}}} \| e^\eps_N(\mu)\|_\cX  < \widetilde{\text{tol}}.
	\end{equation}
	In particular, if multiple selection of snapshots occurs, i.e., if $\mu^{n+1}\in S_n$, then $\overline\Delta_{n}^{\eps}(\mu^{n+1}) < c_\Delta\, \widetilde{\text{tol}}$ and Algorithm \ref{alg:EpsDeltaGreedy} terminates in line \ref{Alg:stop} ensuring \eqref{Eq:final-error}.
\end{proposition}

\begin{remark}
	Obviously, \eqref{eq:surrogate} means that the error measure $\E(\cdot,\cdot)$ must be a rigorous upper bound for the surrogate of the residual-based error estimator $\overline\Delta^\eps_n$ -- at least for the snapshot samples. Hence, \eqref{eq:surrogate} relates the RB-error for $\mu^i\in S_N$ with the snapshot accuracy. Since adaptivity is particularly useful for strongly parameter-dependent problems, we will investigate how to choose $\E(\cdot,\cdot)$ in order to make $C(\mu)$ potentially small, in particular as weakly parameter-sensitive as possible.
\end{remark}

\begin{proof}
	If $\mu^{n+1}\in S_n$, then by line \ref{eps:1} in Algorithm \ref{alg:EpsDeltaGreedy}, we have that $\overline\Delta_{n}^\eps(\mu) \le \overline\Delta_{n}^\eps(\mu^{n+1})$ for all $\mu\in\cD_{\text{train}}$ and that there exists some $1\le i\le n$ such that $\mu^{n+1}=\mu^i$. Then, by \eqref{Eq:dErr}
	\begin{align*}
		\max_{\mu\in \cD_{\text{train}}} \| e^\eps_n(\mu)\|_\cX
		&\le \frac1{c_\Delta} \max_{\mu\in \cD_{\text{train}}} \overline\Delta^\eps_{n}(\mu)
		= \frac1{c_\Delta}\,\overline\Delta_{n}^{\eps}(\mu^{n+1}) \\
		&= \frac1{c_\Delta}\,\overline\Delta_{n}^{\eps}(\mu^{i}) 
		\le \frac{C(\mu^i)}{c_\Delta}\, \E (u(\mu^i), u^\eps(\mu^i)) 
		\,\le\, \frac{C(\mu^i)}{c_\Delta}  \eps(\mu^i)
		\le \widetilde{\text{tol}},
	\end{align*}
	where we used \eqref{Eq:approximate} and the choice of $\eps(\mu)$ in \eqref{Eq:choice-eps}.
\end{proof}

\subsection{Choice of error measure}\label{Sec:ErrorMeasure}
We will now investigate various choices of $\E(\cdot,\cdot)$.

\subsubsection{The True Error}
We start by considering the case $\E(x, \tilde{x}) := \| x-\tilde{x}\|_\cX$. In this case, we can investigate the reproduction error further.

\begin{proposition}\label{prop:snapshotreproductionerror}
Let $b(\cdot, \cdot; \mu): \cX\times\cY\to\R$ be inf-sup stable on $X_{N}^{\eps}\times Y_{N}^{\eps}$ with inf-sup constant $\beta_{N}(\mu)$. 
Moreover, consider the case $\E(x,\tilde x):=\| x-\tilde x\|_\cX$ for the error measure in \eqref{Eq:approximate}. 
Then for all $\mu^{i}\in S_{N}$, we have
\begin{subequations}\label{eq:snapshotreprerr}
	\begin{align}
	\E(u(\mu^i), u^\eps_N(\mu^i)) = 
	\norm{u(\mu^{i}) - u_{N}^{\eps}(\mu^{i})}_{\cX}
	&\leq \frac{\gamma(\mu^{i})}{\beta(\mu^{i})}\,\eps(\mu^{i}),
		\label{eq:snapshotreprerr-a}
	\\ 
	\E(\zeta^\eps_i, u^\eps_N(\mu^i)) = 
	\norm{
		\zeta^\eps_i - u_{N}^{\eps}(\mu^{i})}_{\cX}
	& \leq \frac{\gamma(\mu^{i})}{\beta_N(\mu^{i})}\,\eps(\mu^{i}),
		\qquad \zeta_i^\eps = u^\eps(\mu^i), 
	\label{eq:snapshotreprerr-b}
	\end{align}
\end{subequations}
where $\eps(\mu^{i})$ denotes the accuracy of snapshot $u^{\eps}(\mu^{i})$ in \eqref{Eq:approximate} and $\beta_N(\mu)$ denotes the inf-sup-constant of $b(\cdot,\cdot;\mu)$ on the reduced spaces $X_{N}^{\eps}\subset \cX$, $Y_{N}^{\eps}\subset\cY$.
\end{proposition}
\begin{proof}
Let $\mu^{i}\in S_{N}$. As $X_{N}^{\eps}\subset \cX$, $Y_{N}^{\eps}\subset\cY$, we have Petrov-Galerkin orthogonality w.r.t. the exact solution, \ie,
$b(u(\mu) - u_{N}^{\eps}(\mu), v_{N};\mu) = 0$ for all $v_{N}\in Y_{N}^{\eps}$. This implies the quasi-best approximation property 
$\norm{u(\mu) - u_{N}^{\eps}(\mu)}_{\cX} \leq \frac{\gamma(\mu)}{\beta(\mu)}\inf_{w_{N}\in X_{N}^{\eps}}\norm{u(\mu) - w_{N}}_{\cX}$, {\cite[Thm.\ 2]{MR1971217}}.%
{\footnote{{The original result due to Babu\v{s}ka and Aziz (1972) contains the factor $1+ {\gamma(\mu)}{(\beta(\mu))^{-1}}$. It was shown in \cite[Thm.\ 2]{MR1971217} that the `1+' can be removed.}}} %
As snapshots $\zeta_i^\eps = u^{\eps}(\mu^{i})$, $\mu^{i}\in S_{N}$, are in $X_{N}^{\eps}$, the first inequality \eqref{eq:snapshotreprerr-a} follows with $\inf_{w_{N}\in X_{N}^{\eps}}\norm{u(\mu) - w_{N}}_{\cX}\leq \norm{u(\mu) - \zeta_i^\eps}_{\cX} = \E(u(\mu), u^{\eps}(\mu^{i})) \leq \eps(\mu^{i})$ from \eqref{Eq:approximate}. 

Moreover, we have by choosing $u_{N} = \zeta_i^\eps - u_{N}^{\eps}(\mu^{i}) \in X_{N}^{\eps}$ and using Petrov-Galerkin orthogonality
\begin{align*}
\beta_{N}(\mu^{i}) &\leq \inf_{u_{N}\in X_{N}^{\eps}}\sup_{v_{N}\in Y_{N}^{\eps}}\frac{b(u_{N},v_{N};\mu^{i})}{\norm{u_{N}}_{\cX}\norm{v_{N}}_{\cY}} 
\leq \sup_{v_{N}\in Y_{N}^{\eps}}\frac{b(\zeta_i^\eps - u_{N}^{\eps}(\mu^{i}),v_{N};\mu^{i})}{\norm{\zeta_i^\eps - u_{N}^{\eps}(\mu^{i})}_{\cX}\norm{v_{N}}_{\cY}}\\
&{= \sup_{v_{N}\in Y_{N}^{\eps}}\frac{b(\zeta_i^\eps -  u(\mu^{i}) +  u(\mu^{i}) - u_{N}^{\eps}(\mu^{i}) ,v_{N};\mu^{i})}{\norm{\zeta_i^\eps - u_{N}^{\eps}(\mu^{i})}_{\cX}\norm{v_{N}}_{\cY}}} \\
&= \sup_{v_{N}\in Y_{N}^{\eps}}\frac{b(\zeta_i^\eps - u(\mu^{i}),v_{N};\mu^{i})}{\norm{\zeta_i^\eps - u_{N}^{\eps}(\mu^{i})}_{\cX}\norm{v_{N}}_{\cY}} 
\leq \frac{\gamma(\mu^{i})\, \eps(\mu^{i})}{\norm{\zeta_i^\eps - u_{N}^{\eps}(\mu^{i})}_{\cX}},
\end{align*}
{using continuity in the last step.}
\end{proof}

\begin{corollary}
	In the case $\E(x,\tilde x):=\| x-\tilde x\|_\cX$ assumption \eqref{eq:surrogate} holds with 
	\begin{equation}\label{Eq:Norm:C}
		C(\mu) := C_\Delta \frac{\gamma(\mu)^2}{\beta(\mu)^2}.
	\end{equation}
\end{corollary}
\begin{proof}
	From \eqref{Eq:dErr} and \eqref{eq:snapshotreprerr-a} we get
	\begin{align*}
		\overline\Delta_n^\eps(\mu^i)
		&\le C_\Delta\, \frac{\gamma(\mu^i)}{\beta(\mu^i)}\,
			\| e_n^\eps(\mu^i)\|_\cX
		= C_\Delta\, \frac{\gamma(\mu^i)}{\beta(\mu^i)}\,
			\| u(\mu^i) - u_n^\eps(\mu^i) \|_\cX
			\\
		&\le C_\Delta\, \frac{\gamma(\mu^i)^2}{\beta(\mu^i)^2}\, \eps(\mu^i), 
	\end{align*}
	which proves the claim.
\end{proof}

\paragraph{The elliptic case}\label{Sec:Elliptic}
One might already guess that the factor $\frac{\gamma(\mu)^2}{\beta(\mu)^2}$ in \eqref{Eq:Norm:C} is both overly pessimistic in many cases and computationally demanding since this factor controls the accuracy of the snapshots (in the $\cX$-norm). As we shall see now, this situation can be improved for the elliptic case, i.e., $\cX=\cY$ and $b(\cdot,\cdot;\mu)\equiv a(\cdot,\cdot;\mu)$ being symmetric and coercive with coercivity constant $\alpha(\mu)$. In this case, we may consider the energy norm
$$
	\| w\|_\mu := \sqrt{ a(w,w,;\mu) }, \qquad w\in\cX,\quad \mu\in\cD,
$$
which is equivalent to $\|\cdot\|_\cX$, i.e., $\alpha(\mu)^{1/2} \| w\|_\cX \le \| w\|_\mu \le \gamma(\mu)^{1/2} \| w\|_\cX$. The induced dual norm reads
$$
	\| g\|_{(\mu)'} := \sup_{v\in\cX} \frac{g(v)}{\| v\|_\mu}, \qquad g\in\cX',
$$
so that $\gamma(\mu)^{-1/2} \| g\|_{\cX'} \le \| g\|_{(\mu)'} \le \alpha(\mu)^{-1/2} \| g\|_\cX$. Then, for $\cA(\mu):\cX\to\cX'$ defined as $\langle \cA(\mu)w,v\rangle_{\cX'\times\cX} := a(w,v;\mu)$, $v,w\in\cX$, it is easy to see that $\|\cA(\mu) v\|_{(\mu)'}^2 = a(v,v; \mu)$, so that
\begin{equation}\label{Eq:ell:1}
	\| u(\mu) - u^\eps_N(\mu)\|_\mu \le \| u(\mu) -v_N\|_\mu
	\qquad
	\text{for\ all}\
	v_N\in\cX_N.
\end{equation}
The reason is that the RB solution $u^\eps_N(\mu)$ is the Galerkin projection of $u(\mu)$ onto $X_N$ and a reasoning as in the standard proof of C\'{e}a's lemma shows \eqref{Eq:ell:1}.

\begin{proposition}\label{prop:multiplesnapshots:ell}
	Let $a(\cdot,\cdot;\mu)$ be symmetric and coercive on $\cX$ with coercivity constant $\alpha(\mu)$. In the case $\E(x,\tilde x):=\| x-\tilde x\|_\cX$, assumption \eqref{eq:surrogate} holds with 
	\begin{equation}\label{Eq:choice-eps:ell}
		C(\mu) := C_\Delta \frac{\gamma(\mu)^{3/2}}{\alpha(\mu)^{3/2}}.
	\end{equation}
\end{proposition}

\begin{proof}
	By \eqref{Eq:dErr} and denoting the \emph{residual} as $r_a(v; \mu):=f(\mu) - \cA(\mu) v\in \cX'$ for any $v\in\cX$, we get
	\begin{align}
		\overline\Delta_n^\eps(\mu^i)
		&\le 
		C_\Delta\, \Delta_{n}^{\eps}(\mu^{i})  
		= C_\Delta\, \frac1{\alpha(\mu^i)} \| r_n^\eps(\mu^i)\|_{\cX'}
		\le C_\Delta\,\frac{\sqrt{\gamma(\mu^i)}}{\alpha(\mu^i)} \| r_n^\eps(\mu^i)\|_{(\mu^i)'} \nonumber \\
		&=  C_\Delta\,\frac{\sqrt{\gamma(\mu^i)}}{\alpha(\mu^i)} \sqrt{ a( u(\mu^i)-u^\eps_n(\mu^i), u(\mu^i)-u^\eps_n(\mu^i); \mu^i)} 
					\nonumber\\
		&=  C_\Delta\,\frac{\sqrt{\gamma(\mu^i)}}{\alpha(\mu^i)} \| u(\mu^i)-u^\eps_n(\mu^i) \|_{\mu^i}
					\nonumber\\
		&\le C_\Delta\,\frac{\sqrt{\gamma(\mu^i)}}{\alpha(\mu^i)} \| u(\mu^i)-u^\eps(\mu^i) \|_{\mu^i}
			\label{Eq:ell:2} 
		=C_\Delta\,\frac{\sqrt{\gamma(\mu^i)}}{\alpha(\mu^i)} \| r_a( u^\eps(\mu^i); \mu^i)\|_{(\mu^i)'} 
			\\
		&\le C_\Delta\, \frac{\sqrt{\gamma(\mu^i)}}{\alpha(\mu^i)^{3/2}} \| r_a( u^\eps(\mu^i); \mu^i)\|_{\cX'} 
			\label{Eq:ell:3} \\
		&\le C_\Delta\, 
			\Big(\frac{\gamma(\mu^i)}{\alpha(\mu^i)}\Big)^{3/2} \, \| u(\mu^i) - u^\eps(\mu^i)\|_\cX 
		= C_\Delta\,
			\Big(\frac{\gamma(\mu^i)}{\alpha(\mu^i)}\Big)^{3/2} \, \E(u(\mu^i), u^\eps(\mu^i)),
						\nonumber
	\end{align}
	where \eqref{Eq:ell:2} follows from \eqref{Eq:ell:1} by choosing $v_N=u^\eps(\mu^i)\in X_n$, i.e., the approximate snapshot.
\end{proof}

\begin{remark}
	Obviously, \eqref{Eq:choice-eps:ell} improves upon \eqref{Eq:Norm:C} by a multiplicative factor of $\sqrt{\frac{\alpha(\mu)}{\gamma(\mu)}}$ (and $\alpha(\mu)$ may also be larger than $\beta(\mu)$). 
\end{remark}

\subsubsection{The Residual}

We can further improve the above estimates if we choose a different error measure $\E(\cdot,\cdot)$ in  \eqref{Eq:approximate} for approximating the snapshots, namely the residual, i.e., in the elliptic case $\cA(\mu): \cX\to\cX'$,
	\begin{align}\label{Eq:approximate:ell-res}
		\E(u(\mu), u^\eps(\mu))
		&:= \| r_a(u^\eps(\mu); \mu)\|_{\cX'} 
		= \| \cA(\mu) (u(\mu)- u^\eps(\mu))\|_{\cX'}. 
		\nonumber
	\end{align}
	Since
	$\| r_a(u^\eps(\mu); \mu)\|_{\cX'} \le \gamma(\mu)\, \| u(\mu) - u^\eps(\mu)\|_\cX$, 
	the snapshot accuracy can immediately be relaxed by another factor of $\gamma(\mu)$, which might be significant in some applications.

\begin{corollary}\label{prop:multiplesnapshots:ell-res}
	Let $a(\cdot,\cdot;\mu)$ be symmetric and coercive on $\cX$. In the case $\E(x,\tilde x):=\| \cA(\mu)(x-\tilde{x})\|_{\cX'}$ assumption \eqref{eq:surrogate} holds with 
	\begin{equation}\label{Eq:choice-eps:ell-res-final}
		C(\mu) := C_\Delta \frac{\gamma(\mu)^{1/2}}{\alpha(\mu)^{3/2}}.
	\end{equation}
\end{corollary}
\begin{proof}
	Until \eqref{Eq:ell:3}, we follow the proof of Proposition \ref{prop:multiplesnapshots:ell}, i.e,
	\begin{align*}
		\overline\Delta_n^\eps(\mu^i)
		&\le
		C_\Delta\, 
			\frac{\gamma(\mu^i)^{1/2}}{\alpha(\mu^i)^{3/2}}\, \| r_a( u^\eps(\mu^i); \mu^i)\|_{\cX'} 
		= C_\Delta\, 
			\frac{\gamma(\mu^i)^{1/2}}{\alpha(\mu^i)^{3/2}}\, \E(u(\mu^i), \zeta^\eps_i),
	\end{align*}
	which proves the claim.
\end{proof}

Note, that the improvement of \eqref{Eq:choice-eps:ell-res-final} over \eqref{Eq:choice-eps:ell} or \eqref{Eq:Norm:C} is stronger than at a first glance. In fact, the use of residual instead of the norm of the error incorporates another factor of $\alpha(\mu)$.

However, we can significantly improve the above estimates when we follow the lines of Appendix \ref{App:A}. We consider the normal equation operator operator $\cA(\mu):=\cB^+(\mu)\, \cB(\mu)$ as in Proposition \ref{prop:a3}. Then, $\zeta_i^\eps:=u^\eps(\mu^i)$ is computed as an approximation of $\cA(\mu^i)\, u(\mu^i) = \cB^+(\mu^i)\, f(\mu^i) =: g(\mu^i)$ in $\cX'$. The RB-space is again defined as $X_N^\eps:=\Span\{ u^\eps(\mu^i):\, \mu^i\in S_N\}$. The RB-approximation $u^\eps_N(\mu)\in X_N^\eps$ is then computed as the Galerkin approximation w.r.t.\ the infinite-dimensional normal equation operator $\cA(\mu)$, i.e., 
\begin{equation}\label{Eq:NENbis}
	u^\eps_N(\mu)\in X_N^\eps:\quad
	\langle \cA(\mu)u^\eps_N(\mu), v_N\rangle_{\cX'\times\cX}
	= \langle g(\mu), v_N\rangle_{\cX'\times\cX},
	\quad \forall v_N\in X_N^\eps.
\end{equation}
As shown in Proposition \ref{prop:A6}, this is equivalent to the discrete Petrov-Galerkin problem on $X^\eps_N$ and $Y_N^\eps(\mu) := \cR_\cY'\, \cB(\mu)(X_N^\eps)$, where $\cR_\cY'$ is the adjoint of the Riesz operator $\cR_\cY:\cY'\to\cY$ defined in Definition \ref{def:A1}, i.e., $u^\eps_N(\mu)$ can efficiently be computed as
\begin{equation}\label{Eq:NEN}
	u^\eps_N(\mu)\in X_N^\eps:\quad
	b(u^\eps_N(\mu), w_N; \mu)
	= \langle f(\mu), w_N\rangle_{\cY'\times\cY},
	\quad \forall w_N\in Y_N^\eps(\mu).
\end{equation}
This means that we have a parameter-dependent test space, which -- however -- can be computed online-efficient thanks to the affine decomposition of the bilinear form $b(\cdot,\cdot;\mu)$ w.r.t.\ the parameter $\mu$ in \eqref{eq:affinedecomp}. In fact, in the offline stage, we compute 
	\begin{equation}\label{eq:offline}
		(\eta_{i,q}^\eps,z)_\cY = b^{(q)}(\zeta_i^\eps,z)
		\quad \forall z\in\cY,\, 
		\qquad 1\le q\le Q_b,\, 1\le i\le N,
	\end{equation}
	independent of the parameter, where $\zeta_i^\eps$ again denote the RB-basis functions of $X_N^\eps$. Since the test functions $z\in\cY$ are chosen in the infinite-dimensional space $\cY$, \eqref{eq:offline} amounts to $N\, Q_b$ adaptive solves using \textbf{SOLVE} w.r.t.\ to the Gramian operator of the Hilbert space $\cY$.

	In the online stage, for a given parameter $\mu\in\cD$ (which is not a snapshot), we set $\eta_i^\eps(\mu):=\sum_{q=1}^{Q_b} \theta_b^{(q)}(\mu)\, \eta_{i,q}^\eps$, define
	\begin{equation}\label{eq:PG-Test}
		Y_N^\eps(\mu) := \Span\{ \eta_i^\eps(\mu): \, 1\le i \le N\}
	\end{equation}
and determine $u_N^\eps(\mu)\in X^\eps_N$ by solving $b(u_N^\eps(\mu), v_N; \mu) = f(v_N; \mu)$ for $v_N\in Y_N^\eps(\mu)$. This shows that we have an efficient online-offline separation as in the `classical' RB-case. 
It is readily seen that this choice, which is an  $\infty$-dimensional adaptive analogue of the use of supremizers in the truth spaces, see \cite{MR2281777}, is inf-sup-stable independent of $N$:

\begin{proposition}\label{Prop:LBB}
	For $Y_N^\eps(\mu) := \cR_\cY'\, \cB(\mu)(X_N^\eps)$, we have
	$$
		\inf_{u^\eps_N\in X_N^\eps} \sup_{v_N^\eps(\mu)\in Y_N^\eps(\mu)}
			\frac{b(u_N^\eps, v_N^\eps; \mu)}{\| u_N^\eps\|_\cX\, \| v_N^\eps\|_\cY}
			\ge \beta(\mu) > 0
	$$
	independent of $N$.
\end{proposition}
\begin{proof}
	Let $u_N^\eps\in X_N^\eps$ be arbitrary. Since $Y_N^\eps(\mu) := \cR_\cY'\, \cB(\mu)(X_N^\eps)$, there exists a unique $y_N^\eps(\mu)\in Y_N^\eps(\mu)$ such that $(y_N^\eps(\mu),z)_\cY = b(u_N^\eps, z; \mu)$ for all $z\in\cY$. Using the inf-sup-stability of $b(\cdot,\cdot;\mu)$ yields
	$$
		\beta(\mu) \| u_N^\eps\|_\cX 
		\le \sup_{z\in\cY} \frac{b(u_N^\eps, z; \mu)}{\| z\|_\cY}
		= \sup_{z\in\cY} \frac{(y_N^\eps(\mu),z)_\cY}{\| z\|_\cY}
		= \| y_N^\eps(\mu)\|_\cY.
	$$
	Hence, 
	$$
		\sup_{v_N^\eps(\mu)\in Y_N^\eps(\mu)} \frac{b(u_N^\eps, v_N^\eps; \mu)}{\| v_N^\eps\|_\cY}
		\ge \frac{b(u_N^\eps,  y_N^\eps(\mu); \mu)}{\|  y_N^\eps(\mu)\|_\cY}
		= \| y_N^\eps(\mu)\|_\cY
		\ge \beta(\mu) \| u_N^\eps\|_\cX, 
	$$
	which proves our claim.
\end{proof}

This setting has also an important consequence for the error estimation. Since the RB-solution $u_N^\eps(\mu)$ is an approximate solution of the normal equations \eqref{Eq:NENbis} in $X_N^\eps$ (even though computed as Petrov-Galerkin projection \eqref{Eq:NEN}), Proposition \ref{prop:a4} implies
$$
	\| \cB(\mu^i) ( u(\mu^i) - u_N^\eps(\mu^i))\|_{\cY'}
	=
	\inf_{v_N\in X_N^\eps} 
	\| \cB(\mu^i) ( u(\mu^i) - v_N)\|_{\cY'},
$$
in particular -- since $\zeta_i^\eps\in X_N^\eps$ --
\begin{align*}
	\E(u(\mu^i), u_N^\eps(\mu^i))
	&=
	\| \cB(\mu^i) ( u(\mu^i) - u_N^\eps(\mu^i))\|_{\cY'}\\
	&\le
	\| \cB(\mu^i) ( u(\mu^i) - \zeta^\eps_i)\|_{\cY'}
	= 	\E(u(\mu^i), \zeta^\eps_i).
\end{align*}

\begin{proposition}\label{prop:multiplesnapshots:optimal}
 	Let $\E(x,\tilde x):=\| \cB(\mu)(x-\tilde{x})\|_{\cX'}$ and assume that the RB-solutions are computed via the normal equations \eqref{Eq:NEN}. Then,  \eqref{eq:surrogate} holds with 
	\begin{equation}\label{Eq:choice-eps:ell-res}
		C(\mu^i) := \frac{C_\Delta}{\beta(\mu^i)}.
	\end{equation}
\end{proposition}
\begin{proof}
	By \eqref{Eq:dRes}, we have 
	$$
	\overline\Delta_N^\eps(\mu) 
	\le C_\Delta\, \Delta^\eps_N(\mu)
	= C_\Delta\, \frac{\| r^\eps_{b,N}(\mu)\|_{\cY'}}{\beta(\mu)}
	=  C_\Delta\, \frac{\| \cB(\mu) ( u(\mu) - u^\eps_N(\mu))\|_{\cY'}}{\beta(\mu)}. 
	$$
	For $\mu=\mu^i\in S_N$, we get 
	$$
	\overline\Delta_N^\eps(\mu^i)
	\le C_\Delta\, \frac{\E(u(\mu^i), u_N^\eps(\mu^i))}{\beta(\mu^i)}
	\le \frac{C_\Delta}{\beta(\mu^i)}\, \E(u(\mu^i), \zeta_i^\eps),
	$$
	which proves the claim.
\end{proof}

\begin{remark}
	Obviously, the above estimate significantly improves the previous ones. It also holds in the elliptic case with $\beta(\mu^i)$ replaced by the coercivity constant $\alpha(\mu^i)$.
\end{remark}

We summarize our findings.

\begin{theorem}\label{thm:opt}
	Let  $\widetilde{\text{tol}}>0$ be a given Greedy tolerance, set $\E(x,\tilde x):=\| \cB(\mu)(x-\tilde{x})\|_{\cX'}$ as the error measure in \eqref{Eq:approximate} for the error estimator and assume that the RB-approximations are computed as Petrov-Galerkin solution in \eqref{Eq:NEN}. 
		Then, by setting
	\begin{equation}\label{Eq:choice-eps-final}
		\eps(\mu) := \widetilde{\text{tol}}\, \frac{c_\Delta}{C_\Delta}\,\beta(\mu), 
	\end{equation}
	we have: if Algorithm \ref{alg:EpsDeltaGreedy} terminates for some $N<N_{\max}$, we get the estimate 
	$\max_{\mu\in \cD_{\text{train}}} \| e^\eps_N(\mu)\|_\cX  < \widetilde{\text{tol}}$.
	In particular, if multiple selection of snapshots occurs, Algorithm \ref{alg:EpsDeltaGreedy} terminates in line \ref{Alg:stop}.\hfill\qed
\end{theorem}

\subsubsection{The Normal Equation Residual}\label{Sec:NERes}
If the adaptive algorithm for computing approximate snapshots uses the residual of the normal equation
$$
	r_\cA(w;\mu)
	:= g(\mu) - \cA(\mu)w
	= \cB^+(\mu) (f(\mu)-\cB(\mu)w)
	=: \cB^+(\mu)\, r_b(w;\mu) \in \cX',
$$
i.e., $\E(u(\mu),u^\eps(\mu)):= \| r_\cA(u^\eps(\mu); \mu)\|_{\cX'}$ as stopping criterium, we can easily reformulate the above results. In fact, since
$$
	\| e_N^\eps(\mu)\|_\cX 
	\le \frac1{\beta(\mu)^2}
	\| r_\cA(u_N^\eps(\mu); \mu)\|_{\cX'}
	=: \Delta^\eps_{\cA,N}(\mu),
$$
which is an easy consequence of the fact that $\cA(\mu)$ is coercive with coercivity constant $\beta(\mu)^2$. In a similar way, we define a surrogate $\overline\Delta^\eps_{\cA,N}(\mu)$. Finally, for $\mu^i\in S_N$, we have by Proposition \ref{prop:a4} the relation 
$\| r_\cA(u^\eps_N(\mu^i); \mu^i)\|_{\cX'}
\le
\| r_\cA(\zeta_\eps^i; \mu^i)\|_{\cX'}$, 	
so that
$$
	\eps(\mu) = \widetilde{\text{tol}}\, \frac{c_\Delta}{C_\Delta}\,\beta(\mu)^2
$$
is the appropriate choice for the snapshot tolerance in this case, i.e., one gets another multiplicative factor of $\beta(\mu)$.


\section{{Adaptive Wavelet Galerkin Methods (AWGMs)}} \label{sec:awgm}
To obtain an adaptive approximation for the  snapshots $u^\cN(\mu)$ as well as the error estimators $\Delta_N^\cN(\mu)$ we employ \emph{adaptive wavelet Galerkin methods (AWGMs)} that have first been introduced in \cite{CDD01, CDD02} for stationary problems and extended to space-time variational parabolic problems in \cite{SchwabStevenson}. 
We will also use wavelet methods to construct a computable approximate error estimator $\overline\Delta_{N}^\eps(\mu)$ as in \S\ref{sec:equiverrorest}, \eqref{Eq:dRes}.

For the AWGM, we used \emph{multitree-based} versions developed in \cite{Kestler:Diss, MT-LS-AWGM,Kestler:2012d,Kestler:2012c}, which we briefly review. Let $\cA:\cX\to\cY'$ be a linear differential (or integral) operator which may or may not depend on $\mu\in\cD$. Given some $b\in\cY'$, we look for $x\in\cX$ such that
\begin{equation}\label{eq:GP}
	\cA x = b \quad\text{in }\, \cY'.
\end{equation}

\subsection{Equivalent Bi-infinite Matrix-Vector Problem}
Variational equations of the form \eqref{eq:GP} can be reformulated as \emph{equivalent} $\ell_{2}$-problems by considering \emph{Riesz bases} of the Hilbert spaces $\cX$, $\cY$. We call $\Upsilon := \{ \gamma_i : i \in \N \}\subset\cZ$ a \emph{Riesz basis} for a separable Hilbert space $\cZ$ if its linear span is dense in $\cZ$ and  if there exist $\mathrm{c},\mathrm{C}>0$ such that 
\begin{equation} \label{eq:riesz_basis}
	\mathrm{c}\| \bv \|^2_{\ell_2(\N)} 
	\leq \| v \|^2_{\cZ} 
	\leq \mathrm{C} \| \bv \|^2_{\ell_2(\N)} 
	\quad \forall \bv = (v_i)_{i \in \N} \in \ell_2(\N),\ v = \sum_{i=1}^\infty v_i \gamma_i.
\end{equation}
For $\cX$, $\cY$, we denote these Riesz wavelet bases by 
\begin{equation} \label{eq:Psis_cXcY}
   	\hatbPsi^{\cX} := \big\{ \hatbpsi^{\cX}_{\blambda} : \blambda \in \hatbcJ \big\} \subset \cX, \qquad 
	\checkbPsi^{\cY} := \big\{ \checkbpsi^{\cY}_{\blambda} : \blambda \in \checkbcJ \big\} \subset \cY,
\end{equation}
for countable index sets $\hatbcJ$, $\checkbcJ$. Such bases can be constructed by first building univariate \emph{wavelet bases} $\Psi=\{ \psi_\lambda:\, \lambda\in\cJ\}$ for $L_{2}(0,1)$ that are sufficiently smooth to constitute (after a proper normalization) also Riesz bases for a whole range of Sobolev spaces $H^s(0,1)$, $s\in (-\tilde\gamma, \gamma)$, where $\gamma, \tilde\gamma>0$ depend on the choice of the wavelets, \confer \cite{Urban:WaveletBook}. Typically the index takes the form $\lambda=(j,k)$, where $|\lambda|:=j$ denotes the \emph{level} (e.g., $|\supp\, \psi_\lambda|\sim 2^{-|\lambda|}$) and $k$ the location in $(0,1)$, e.g., the center of its support. We consider piecewise polynomial wavelets of order $d$ (degree plus one). Wavelets are oscillating (``small waves'') which is reflected by their degree $m$ of \emph{vanishing moments}, i.e., 
$\int_0^1 x^r \psi_\lambda(x)\, dx=0$ for all $|\lambda|>0$ and all $0\le r\le m-1$, where $|\lambda|=0$ denotes the coarsest level, $0=\min_{\lambda\in\cJ} |\lambda|$. Those functions are no `true' wavelets but, e.g., splines (scaling functions). The above mentioned constants $\gamma$ and $\tilde\gamma$ are determined by $d$, $m$ and $\tilde d$, $\tilde m$, which are the corresponding parameters of the \emph{dual} wavelet basis $\tilde\psi=\{ \tilde\psi_\lambda:\, \lambda\in\cJ\}$ with $\int_0^1 \psi_\lambda(x)\, \tilde\psi_\lambda(x)\, dx=\delta_{\lambda, \tilde\lambda}$ for all $\lambda, \tilde\lambda\in\cJ$ with $|\lambda|, |\tilde\lambda|>0$.

Tensorization of the univariate functions then allows for appropriate bases in higher dimensions as well as for a vast range of Bochner spaces arising in the formulation of parabolic PDEs, see, \eg, \cite{SchwabStevenson}. Constructions for more complicated domains $\Omega$ are also available.

Then, we equivalently formulate \eqref{eq:GP}  as the discrete, but infinite-dimensional equation
\begin{equation} \label{eq:equiv_ell2_prob}
  \text{Find }\bx \in \ell_2(\hatbcJ): \qquad 
  	\bA \bx = \bb,
	\qquad \bb \in \ell_2(\checkbcJ),
\end{equation}
where $\bA := \eval{\checkbPsi^{\cY}}{\cA[\hatbPsi^{\cX}]}$, 
$\bb = \big[ b(\checkbpsi^{\cY}_{\blambda}) \big]_{\blambda \in \checkbcJ}$ and $\bx$ are the coefficients of the (unique) expansion $x = \bx^{\top}\hatbPsi^{\cX}$.

\subsection{Adaptive Methods and Nonlinear Approximation}
In order to approximately solve the infinite-dimensional equation \eqref{eq:equiv_ell2_prob}, AWGMs iteratively construct a sequence of \emph{nested finite} index sets $(\hatbLambda_{k})_{k} \subset \hatbcJ$, $(\checkbLambda_{k})_{k} \subset \checkbcJ$, to which \eqref{eq:equiv_ell2_prob} is restricted. 
Considering (just for ease of presentation) a linear self-adjoint operator $\cA:\cX \to\cX$ and $\hatbPsi = \hatbPsi^{\cX} = \checkbPsi^{\cY}$, in each iteration the finite-dimensional problem
\begin{equation} \label{eq:galerkin_system_elliptic}
 	\text{Find } \bx_{\hatbLambda_k} \in \ell_2(\hatbLambda_k): 
	\qquad {}_{\hatbLambda_k}\!\!\bA_{\!\hatbLambda_k}\, \bx_{\hatbLambda_k} = \bb_{\hatbLambda_{k}}, 
	\quad \bb_{\hatbLambda_k} \in \ell_2(\hatbLambda_k),
\end{equation}
 is solved, where for $\bLambda \subset \bcJ$, $\bv_{\bLambda} := \bv|_{\bLambda}$ denotes the restriction of $\bv \in \ell_{2}(\bcJ)$ to $\ell_{2}(\bLambda)$ and ${}_{\bLambda}\!\bA_{\!\bLambda} := (\bA \bE_{\bLambda})|_{\bLambda}$ with trivial embedding $\bE:\ell_{2}(\bLambda) \to \ell_{2}(\bcJ)$ the restriction of $\bA$ in both rows and columns.
 
The extension of $\bLambda_{k}$ to $\bLambda_{k+1}$ is then based on the residual $\br^{k}  :=\bb - \bA \bx_{\bLambda_{k}}$ and its norm $\norm{\br^{k}}_{\ell_{2}(\bcJ)}$ which forms an equivalent error estimator, since
\begin{equation} \label{eq:estimate_elliptic_problem}
    \| \bA \|^{-1} \| \br^{k} \|_{\ell_2(\hatbcJ)} 
    		\leq \|   \bx -  \bx_{\hatbLambda_k} \|_{\ell_2(\hatbcJ)}  
		\leq  \| \bA^{-1} \| \| \br^{k} \|_{\ell_2(\hatbcJ)}.
\end{equation}
Note that $\br^{k}$ is supported on the \emph{infinite-dimensional} set $\hatbcJ$ even if $\bx_{\bLambda_{k}}$ is finitely supported. Hence we have to use appropriate approximation methods for the residual evaluation in order to arrive at an implementable AGWM, see \S \ref{subsec:MTImpls} below.

The next index set is obtained by a so-called \emph{bulk-chasing}: choose $\bLambda_{k+1} \supset \bLambda_{k}$ as the smallest index set such that $\norm{\br^{k}_{\bLambda_{k+1}}}_{\ell_{2}(\bLambda_{k+1})} \geq c \norm{\br^{k}}_{\ell_{2}(\bcJ)}$ for some $0 < c < 1$. This implies that the indices of the largest residual coefficients are added to $\bLambda_{k}$ and the adaptive index set is steered into the direction of the largest error. 

Under appropriate assumptions on the exactness and computational cost of the solution of \eqref{eq:galerkin_system_elliptic}, the approximation of $\br^{k}$ and the implementation of the bulk chasing process, a \emph{quasi-optimality} result is known. In order to formulate it, we introduce the nonlinear approximation class {(recall that $\hatbPsi$ is a Riesz basis)}
\begin{equation}\label{eq:DefAs}
	\fA^s :=\! \big\{  \bv \in \ell_2(\hatbcJ) \!:\! \| \bv \|_{\fA^s} 
	\!:=\! \sup_{\eps >0} \eps \!\cdot\! \big[  \min\{ \cN \in \N_0 \!:\! \| \bv - \bv_\cN \|_{\ell_2(\hatbcJ)} \leq \eps  \}  \big]^s \!<\! \infty  \big\}
\end{equation}
with $\bv_{\cN}$ being the best $\cN$-term approximation on $\bv$, consisting of the $\cN$ largest coefficients in modulus of $\bv$.

\begin{theorem}[{\confer \cite{Gantumur:2007,Stevenson:2009}}]\label{theo:quasioptimalAWGM}
There exist implementable routines and parameters such that the (approximate) computations of $\bx_{\hatbLambda_{k}}$, $\br^{k}$ and $\hatbLambda_{k+1}$ can be performed with controllable tolerances and computational cost: 
if the AWGM is terminated when $\norm{\br^{k}_{\hatbLambda_{k}}}_{\ell_{2}(\hatbLambda_{k})} \leq \eps / \norm{\bA^{-1}}$, 
the output $\bx_{\eps} := \bx_{\hatbLambda_{k}}$ satisfies $\| \bx - \bx_{\eps} \|_{\ell_2(\hatbcJ)} \leq \eps$.
If, moreover, $\bx \in \fA^s$ for some $s>0$, it holds for $\cN_k := \# \hatbLambda_k$ that
\begin{equation} \label{eq:quasi_optimal_awgm}
     \| \bx - \bx_{\eps} \|_{\ell_2(\hatbcJ)} 
     \le C\, \| \bx \|^{1/s}_{\fA^s} \cN_k^{-s} , 
     \qquad \# \supp \bx_{\eps} \le C\, \eps^{-1/s} \| \bx \|^{1/s}_{\fA^s}.
\end{equation}
If $s$ is small enough, the computation of $\bx_{\eps}$ can be realized with a computational cost that is bounded by an absolute multiple of $\eps^{-1/s} \| \bx \|^{1/s}_{\fA^s}$, i.e., linear complexity.{\footnote{{This notion means that the solution can be computed with cost which is in the order of the number of unknowns, recall the second estimate in \eqref{eq:quasi_optimal_awgm}.}}}\qed
\end{theorem}
Theorem \ref{theo:quasioptimalAWGM} states that AWGMs are quasi-optimal in the sense that the optimal convergence rate for best $\cN$-term approximations of $\bx$ can be realized up to some constant within linear computational complexity. These techniques can be extended to problems that are neither symmetric nor positive-definite by considering the normal equations $\bA^{\!\top}\!\bA\bx = \bA^{\!\top}\! \bb$. This includes Petrov-Galerkin problems as they arise, \eg, in space-time formulations of parabolic PDEs, even if the wavelet bases $\hatbPsi^{\cX}$, $\checkbPsi^{\cY}$ for $\cX$ and $\cY$ differ not only in scaling but are even obtained from different sets of wavelets \cite{Chegini:2011, MT-LS-AWGM}.

\subsection{Multitree-based Implementations}\label{subsec:MTImpls}
Several different implementations of quasi-optimal AWGMs have been proposed. The algorithms in \cite{CDD01, CDD02} use a \emph{thresholding step} in order to retrieve the optimal computational complexity in Theorem \ref{theo:quasioptimalAWGM}, which in the case of \cite{CDD02} is combined with an inexact Richardson iteration on the infinite-dimensional equation \eqref{eq:equiv_ell2_prob}. In \cite{Gantumur:2007} a residual approximation method is employed that does not require thresholding and can thus be proven to be more efficient. However, like the afore-mentioned algorithms it relies on the application of a so-called \textbf{APPLY} routine in order to approximate the arising infinite-dimensional matrix-vector products $\bA \bv \in \ell_{2}(\hatbcJ)$. Such routines are based on \emph{wavelet compression schemes}, require certain characteristics of the wavelet bases as well as compressibility results for the operator $\cA$ and are in general quantitatively demanding. 
For these reasons, we employ \emph{multitree-based} matrix-vector product evaluations in the solution of \eqref{eq:galerkin_system_elliptic} and the approximation of the residual $\br^{k}$, as proposed in \cite{Kestler:2012d,Kestler:2012c}. That is, we restrict the index sets $\hatbLambda_{k}$ to multitrees in the sense of the following definition.

\begin{definition}
(i) For a univariate uniformly local, piecewise polynomial wavelet basis $\Psi = \{ \psi_\lambda : \lambda \in \cJ\}$, a set $\Lambda \subset \cJ$ is called a \emph{tree} if for any $\lambda \in \Lambda$ with $|\lambda|>0$ it holds that $\supp\,\psi_\lambda \subset \bigcup_{\mu\in\Lambda; |\mu|=\lambda-1} \supp\,\psi_\mu$. 
(ii) An index set $\bLambda \in \bcJ$ belonging to a tensor product wavelet basis $\bPsi = \{ \bpsi_{\lambda} : \blambda \in \bcJ \}$ is  called a \emph{multitree} if for all $i \in \{ 0,\ldots,n\}$ and all indices  $\mu_j \in \cJ^{(j)}$ for  $j \neq i$,  the index set
\begin{equation} \label{eq:tree_from_multitree}
   \Lambda^{(i)} := \{ \lambda_i \in \cJ^{(i)} :  (\mu_0, \ldots,\mu_{i-1},\lambda_i,\mu_{i+1},\ldots,\mu_n) \in \bLambda\}    \subset \cJ^{(i)}
\end{equation}
is either the empty set or a tree.\qed
\end{definition}

The restriction to such index sets preserves the quasi-optimality of the AWGM \cite{Kestler:2012d} in the constrained approximation class $\cAmtree^s$ defined w.r.t. $\| \bv \|_{\cAmtree^s} := \sup_{\eps >0} \eps \cdot \big[  \min\{ \cN \in \N_0: \| \bv - \bv_\cN \|_{\ell_2(\hatbcJ)} \leq \eps \; \wedge \; \supp \bv_{\cN} \text{ is a multitree} \} \big]^s$
and allows a computationally very efficient evaluation of finite-dimensional matrix-vector products:
\begin{theorem}[{\cite[Theorem 3.1]{Kestler:2012c}}] \label{thm:mv_mt}
	Let $\cA$ be a linear differential operator with polynomial coefficients and {let} $\hatbLambda \subset \hatbcJ$, $\checkbLambda \in \checkbcJ$ be multitrees.
Then, for any $\bv_{\hatbLambda} \in \ell_2(\hatbLambda)$, the product ${}_{\checkbLambda} \bA_{\hatbLambda}\, \bv_{\hatbLambda}$ can be {computed} in $\cO(\# \hatbLambda + \# \checkbLambda)$ operations.
\end{theorem}

Moreover, we obtain the following approximation result for the residual:
\begin{theorem}[{\cite{Kestler:2012d}}] \label{thm:primal_res}
	Let $0< \omega <1$, {let} $\cA$ be a differential operator with polynomial coefficients and {let} $\bx \in \cAmtree^s$ for some $s>0$.
	Then,  for all finite multitrees $\hatbLambda \subset \hatbcJ$ and {all} $\bw_{\hatbLambda} \in \ell_2(\hatbLambda)$, there exists a multitree $\checkbXi = \checkbXi(\hatbLambda,\omega) \subset \checkbcJ$ such that for $\br := \bb_{\checkbXi} - {}_{\checkbXi} \bA_{\hatbLambda}\, \bw_{\hatbLambda}$ it holds that $\# \checkbXi \le C\, \# \hatbLambda + \| \br \|_{\ell_2(\checkbcJ)}^{-1/s}$  and
\begin{equation} \label{eq:error_estim_primal_res}
    \| (\bb - \bA \bw_{\hatbLambda}) - \br \|_{\ell_2(\checkbcJ)} \leq \omega \| \br \|_{\ell_2(\checkbcJ)}.
\end{equation}
\end{theorem}

Thus, the computational cost for the residual approximation is of the order $\cO(\# \hatbLambda + \| \br \|_{\ell_2(\checkbcJ)}^{-1/s})$ if the right hand side coefficients $\bb_{\checkbXi}$ can be computed efficiently.  Explicit constructions of $\checkbXi$ are discussed in \cite{Kestler:2012d} and \cite{MT-LS-AWGM}, where the multitree-based AWGM is extended to the normal equations. In particular, such AWGM satisfies the conditions posed for the routine \textbf{SOLVE} in Section \ref{sec:adaptiverbm}. We used AWGM for all adaptive computations (snapshots, supremizers, error estimates).

\subsection{Wavelet-based adaptive residual RB-error estimate}\label{Sec:WavEst}
Recall from \eqref{Eq:Delta-eps-N} the definition of the error estimator,
$$
	\Delta_N^\eps(\mu)
	= \frac{R^\eps_{b,N}(\mu)}{\beta(\mu)}
	= \frac{\| r^\eps_{b,N}(\mu)\|_{\cY'}}{\beta(\mu)}
	= \frac{\| f(\mu) - \cB(\mu)\, u_N^\eps(\mu)\|_{\cY'}}{\beta(\mu)},
$$
where $f(\mu) := f(\cdot;\mu)\in\cY'$ and $\cB(\mu):\cX\to\cY'$ is defined as in \S \ref{Sec:2.1}. If we assume that an efficiently computable lower bound $0<\beta_{\text{LB}}(\mu) \le \beta(\mu)$ for the inf-sup-constant is available (e.g., by the Successive Constraint Method -- SCM --, see \cite{MR2367928}), we are left with the problem of approximating $R_{b,N}^\eps(\mu)$, the dual norm of the residual.

Let us now show how this can be done in an online-efficient manner using the online-offline decomposition combined with the wavelet expansion. Using \eqref{eq:affinedecomp} yields
\begin{align}
	r_{b, N}^\eps(\mu)
	= f(\mu) - \cB(\mu)\, u_N^\eps(\mu) 
	&= \sum_{q=1}^{Q_f} \theta_f^{(q)}(\mu) f^{(q)} - \sum_{q=1}^{Q_b} \theta_b^{(q)}(\mu)\, B^{(q)} u_N^\eps(\mu) 
		\nonumber\\
	&= \sum_{q=1}^{Q_f} \theta_f^{(q)}(\mu) f^{(q)} 
		- \sum_{i=1}^N \sum_{q=1}^{Q_b} u_i^N(\mu) \theta_b^{(q)}(\mu)\, B^{(q)} \zeta_i^\eps, 
	\label{eq:res-affine}
\end{align}
with $\zeta_i^\eps$ defined in \eqref{eq:AdaptiveRBSpace} and $B^{(q)}:\cX\to\cY'$ defined by $\langle B^{(q)} w, v\rangle_{\cY'\times\cY} := b^{(q)}(w,v)$, $w\in\cX$, $v\in\cY$.

Next, recall from \eqref{eq:Psis_cXcY} that $\checkbPsi^{\cY} = \big\{ \checkbpsi^{\cY}_{\blambda} : \blambda \in \checkbcJ \big\}$ is a Riesz basis for $\cY$. Then, from the Riesz representation theorem, it is well-known that a dual wavelet system $\tilde\checkbPsi^{\cY}= \big\{ \tilde{\checkbpsi}^{\cY}_{\blambda} : \blambda \in \checkbcJ \big\}$ exists which is a Riesz basis for the dual space $\cY'$. Let $g\in\cY'$, then this element has a unique expansion in the dual wavelet basis, i.e.,
$$
	g = \sum_{\blambda \in \checkbcJ} g_{\blambda} \tilde{\checkbpsi}^{\cY}_{\blambda},
	\qquad
	\bg := (g_{\blambda})_{\blambda \in \checkbcJ},
	\quad
	g_{\blambda} = \langle g, \checkbpsi^{\cY}_{\blambda} \rangle_{\cY'\times\cY}.
$$
In particular, the wavelet coefficients $g_{\blambda}$ are computed by the dual pairing of $g$ with the \emph{primal} wavelets, which are often piecewise polynomials, so that the arising integrals can efficiently be computed at any desired accuracy.

The Riesz basis property implies the existence of constants $0< c_\Psi\le C_\Psi<\infty$ such that for all $g\in\cY'$ it holds
\begin{equation}\label{eq:dualnorm-eq}
	c_\Psi\, \| g\|_{\cY'}
	\leq \Big(\sum_{\blambda \in \checkbcJ} |g_{\blambda}|^2\Big)^{1/2} 
	= \| \bg\|_{\ell_2(\checkbcJ)}
	\le C_\Psi\, \| g\|_{\cY'},
\end{equation}
where the equivalence constants $c_\Psi$ and $C_\Psi$ depend only on the choice of $\checkbPsi^{\cY}$. Putting \eqref{eq:res-affine} and \eqref{eq:dualnorm-eq} together yields
\begin{align*}
	\| r_{b,N}^\eps(\mu)\|_{\cY'}^2
	&\le c_\Psi^{-1} \sum_{\blambda \in \checkbcJ} 
		\big( \langle f(\mu) - \cB(\mu)\, u_N^\eps(\mu), \checkbpsi^{\cY}_{\blambda} \rangle_{\cY'\times\cY}\big)^2 \\
	&\kern-35pt	
		= c_\Psi^{-1} \sum_{\blambda \in \checkbcJ} 
		\Big(
		\sum_{q=1}^{Q_f} \theta_f^{(q)}(\mu) \langle f^{(q)} , \checkbpsi^{\cY}_{\blambda} \rangle_{\cY'\times\cY}
		- \sum_{i=1}^N \sum_{q=1}^{Q_b} u_i^N(\mu) \theta_b^{(q)}(\mu)\, b^{(q)} (\zeta_i^\eps, \checkbpsi^{\cY}_{\blambda})
		\Big)^2 \\
	&\kern-35pt	
		= c_\Psi^{-1} \sum_{q, q'=1}^{Q_f} 
			\theta_f^{(q)}(\mu) \theta_f^{(q')}(\mu)
				\sum_{\blambda \in \checkbcJ} 
					\langle f^{(q)} , \checkbpsi^{\cY}_{\blambda} \rangle_{\cY'\times\cY}\,
					\langle f^{(q')} , \checkbpsi^{\cY}_{\blambda} \rangle_{\cY'\times\cY} \\
	&\kern-30pt	
		+ c_\Psi^{-1} \sum_{q, q'=1}^{Q_b} \sum_{i,j=1}^N
			\theta_b^{(q)}(\mu) \theta_b^{(q')}(\mu)\,
			u_i^N(\mu)\, u_j^N(\mu) 
				\sum_{\blambda \in \checkbcJ} 
					b^{(q)} (\zeta_i^\eps, \checkbpsi^{\cY}_{\blambda})\, b^{(q')} (\zeta_j^\eps, \checkbpsi^{\cY}_{\blambda})\\
	&\kern-30pt	
		- 2 c_\Psi^{-1}  \sum_{q=1}^{Q_f} \sum_{q'=1}^{Q_b} \sum_{j=1}^N
			\theta_f^{(q)}(\mu)\, \theta_b^{(q')}(\mu)\, u_j^N(\mu) 
				\sum_{\blambda \in \checkbcJ} 
					\langle f^{(q)} , \checkbpsi^{\cY}_{\blambda} \rangle_{\cY'\times\cY}\,
					b^{(q')} (\zeta_j^\eps, \checkbpsi^{\cY}_{\blambda}) \\
	&\kern-35pt	
		= c_\Psi^{-1} \sum_{q, q'=1}^{Q_f} 
			\theta_f^{(q)}(\mu) \theta_f^{(q')}(\mu)
				C^{f,f}_{q,q'} \\
	&\kern-30pt	
		+ c_\Psi^{-1} \sum_{q, q'=1}^{Q_b} \sum_{i,j=1}^N
			\theta_b^{(q)}(\mu) \theta_b^{(q')}(\mu)\,
			u_i^N(\mu)\, u_j^N(\mu) 
				C^{b,b}_{(i,q),(j,q')}\\
	&\kern-30pt	
		- 2 c_\Psi^{-1}  \sum_{q=1}^{Q_f} \sum_{q'=1}^{Q_b} \sum_{j=1}^N
			\theta_f^{(q)}(\mu)\, \theta_b^{(q')}(\mu)\, u_j^N(\mu) 
				C^{f,b}_{q,(j,q')},
\end{align*}
where the terms
\begin{align*}
	C^{f,f}_{q,q'} 
		&:= 
				\sum_{\blambda \in \checkbcJ} 
					\langle f^{(q)} , \checkbpsi^{\cY}_{\blambda} \rangle_{\cY'\times\cY}\,
					\langle f^{(q')} , \checkbpsi^{\cY}_{\blambda} \rangle_{\cY'\times\cY}, \\
	C^{b,b}_{(i,q),(j,q')}
		&:=
				\sum_{\blambda \in \checkbcJ} 
					b^{(q)} (\zeta_i^\eps, \checkbpsi^{\cY}_{\blambda})\, b^{(q')} (\zeta_j^\eps, \checkbpsi^{\cY}_{\blambda}), \\
	C^{f,b}_{q,(j,q')}
		&:= 
				\sum_{\blambda \in \checkbcJ} 
					\langle f^{(q)} , \checkbpsi^{\cY}_{\blambda} \rangle_{\cY'\times\cY}\,
					b^{(q')} (\zeta_j^\eps, \checkbpsi^{\cY}_{\blambda}) 
\end{align*}
can be computed offline in principle exactly -- or at least up to any desirable accuracy, which can be seen as follows: in principle, the index set $\checkbcJ$ has infinitely many elements, so that all three sums have infinitely many terms. However,
\begin{compactitem}
	\item $f^{(q)}$, $1\le q\le Q_f$, are given elements in $\cY'$. Either they have a finite wavelet expansion (and then both $C^{f,f}_{q,q'}$ and $C^{f,b}_{q,(j,q')}$ are finite sums) or at least the sequence of wavelet coefficients decay with respect to the level (the $\cY'$-norm is finite and the sum has to converge). In this case, both $C^{f,f}_{q,q'}$ and $C^{f,b}_{q,(j,q')}$ can be truncated and the desired accuracy triggers the number of terms in this offline computation;
	\item $b^{(q)}(\zeta_i^\eps, \cdot)$, $1\le q\le Q_b$, $1\le i\le N$, are also given functionals in $\cY'$, so that the same reasoning as above applies for the sum in $C^{b,b}_{(i,q),(j,q')}$.
\end{compactitem}
The number of terms in these expansions as well as their localization have a strong influence on the decision if an adaptive snapshot computation is indeed required or if, e.g., an adaptively generated common truth as in \cite{Masa:New} might be sufficient. Details concerning the decay of wavelet coefficients can be found in \cite{CDD01,Stevenson:2009,Urban:WaveletBook}.

In summary, the online complexity is $\cO(Q_f^2+Q_f\, Q_b\, N + Q_b^2\, N^2)$, i.e., the surrogate for the error estimator can be computed online-efficient.

\section{{Numerical Experiments}} \label{sec:numerics}

In this section, we present numerical results showing quantitative effects of an adaptive offline computation of the snapshots as well as of the adaptive wavelet computation of the dual norm of the residual. We recall that meaningful test cases have to be strongly parameter-dependent, so that the presented results need to be properly interpreted. As numerical examples, we consider an elliptic problem as well as a parabolic (time-periodic) one in space-time formulation.

\subsubsection*{Realization of the Error Estimator}
As we have seen in Section \ref{Sec:WavEst}, the terms $b^{(q')} (\zeta_j^\eps, \checkbpsi^{\cY}_{\blambda})$ and $\langle f^{(q)} , \checkbpsi^{\cY}_{\blambda} \rangle_{\cY'\times\cY}$ need to be computed in the offline stage. If the data do not allow for a finite wavelet expansion (which is the case in our example), the corresponding wavelet expansions need to be truncated. The corresponding error can be controlled by the size of the wavelet coefficients, e.g., \cite{Kestler:Diss,Urban:WaveletBook}. Next, the (in principle) infinite sums in $C^{f,f}_{q,q'}$, $C^{b,b}_{(i,q),(j,q')}$ and $C^{f,b}_{q,(j,q')}$ in Section \ref{Sec:WavEst} have to be computed, which can be done at any desired accuracy offline due to the decay of the wavelet coefficients w.r.t.\ the level. Of course, the terms $b^{(q')} (\zeta_j^\eps, \checkbpsi^{\cY}_{\blambda})$ have to be computed after the adaptive computation of the corresponding snapshot. For our experiments, we have chosen a sufficiently high maximal level for both the error estimator and the exact problem.

Finally, we need the constants $c_\Psi$ and $C_\Psi$ in \eqref{eq:dualnorm-eq}, namely the Riesz constants of the wavelet basis. These numbers can either be taken from the literature or by determining smallest and largest eigenvalues of the dual mass operator $(\tilde\checkbPsi^{\cY}, \tilde\checkbPsi^{\cY})_\cY$.

\subsection{An Elliptic Equation: A Thermal Block with a Seal}\label{Sec:Ex1}

\subsubsection{Data}
We consider heat conduction in a 2D thermal block $\Omega=(0,1)^{2}$ consisting of two subdomains $\Omega_{0}=[0.5,1]\times [0,1]$, $\Omega_{1}=[0,0.5]\times [0,1]$, with different conductivities $\mu_{0}=1$, $\mu_{1} \in [0.01,100]$, \cite{RozzaRBMIntro}. The heat influx is modeled as a constant local source on different parts $\widetilde \Omega_{i}$, $i=1,\dots,9$, of the domain, where the current location depends on a (discrete) parameter $\mu_{2}\in\{ 1,\ldots ,9\}$, see Figure \ref{Fig:Block}.  We impose homogeneous Dirichlet boundary conditions on $\Gamma_{D}:= \partial\Omega \cap \{ x = 0 \vee x = 1\}$ and homogeneous Neumann conditions on $\Gamma_{N}:= \partial\Omega \cap \{y=0 \vee y=1\}$. The variational formulation then reads: find $u \in \cX := H^1_D(\Omega) = \{ v \in H^{1}(\Omega) : v = 0 \text{ on } \Gamma_{D}\}$ such that 
\begin{equation*}
	 \int_{\Omega_{0}} \nabla u \cdot \nabla v + \mu_{1} \int_{\Omega_{1}} \nabla u \cdot \nabla v = \left(f(\mu_{2}), v\right)_{L_2(\Omega)} \,		 
	 	\forall\, v \in \cX,
	 \quad
	 f(\mu_{2}) := {\sum_{i=1}^{9}} \delta_{\mu_2,i} \ind{\widetilde\Omega_{i}}.
\end{equation*}
\begin{wrapfigure}{l}{4.3cm}			
	\begin{tikzpicture}[scale=0.9]
        			\draw[thick] (0,0) grid[xstep=1.5, ystep=3] (3,3);
        		    	\node[color=gray] at (1.10,1.2) {\huge$\Omega_{1}$};
        		    	\node[color=gray]  at (2.10,1.2) {\huge$\Omega_{0}$};
        		    	\draw[thin, dashed] (0,0) grid[xstep=1,ystep=1.2] (3,3);
        		    	\node at (0.5, 0.6) {$\widetilde\Omega_{1}$};
        		    	\node at (0.5, 1.65) {$\widetilde\Omega_{2}$};
        		    	\node at (0.5, 2.7) {$\widetilde\Omega_{3}$};
        		    	\node at (1.5, 0.6) {$\widetilde\Omega_{4}$};
        		    	\node at (1.5, 1.65) {$\widetilde\Omega_{5}$};
        		    	\node at (1.5, 2.7) {$\widetilde\Omega_{6}$};
        		    	\node at (2.5, 0.6) {$\widetilde\Omega_{7}$};
        		    	\node at (2.5, 1.65) {$\widetilde\Omega_{8}$};
        		    	\node at (2.5, 2.7) {$\widetilde\Omega_{9}$};
        		    	\node[color=gray] at (1.0, -0.3) {{\tiny$\frac13$}};
        		    	\node[color=gray] at (2.0, -0.3) {{\tiny$\frac23$}};
        		    	\node[color=gray] at (-0.15, 1.2) {{\tiny$\frac25$}};
        		    	\node[color=gray] at (-0.15, 2.4) {{\tiny$\frac45$}};
        		    	\draw[thick] (0,0) -- (3,0);
        		    	\draw[thick] (0,3) -- (3,3);
        		    	\draw[dashed, thick] (0,0) -- (0,3);
        		    	\draw[dashed, thick] (3,0) -- (3,3);
        			\draw[thick] (0,0) -- (3,0) node[midway, below]{$\Gamma_{N}$};
        			\draw[thick] (0,0) -- (0,3) node[midway, left]{$\Gamma_{D}$};
        			\draw[thick] (0,3) -- (3,3) node[midway, above]{$\Gamma_{N}$};
        			\draw[thick] (3,0) -- (3,3) node[midway, right]{$\Gamma_{D}$};
			\end{tikzpicture}
		\caption{\label{Fig:Block}%
		Thermal block with $9$ local sources.
		}
\end{wrapfigure}
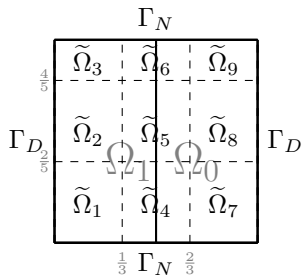

\subsubsection{Wavelet Discretization}
We employ a multi\-tree-based AWGM (see Section \ref{sec:awgm}) with a tensor basis consisting of bi-orthogonal B-spline wavelets from \cite{Dijkema:Diss} of order $d_{x}= m_{x}=2$ (for the meaning of the parameters $d$ and $m$, see Section \ref{sec:awgm}) and $L_{2}(0,1)$-orthonormal (multi-)wavelets as in \cite{Rupp:Diss} of order $d_{y}=m_y=2$, with homogeneous boundary conditions. In order to show that adaptive discretizations are beneficial in this case, we indicate in Figure \ref{fig:support} the support centers of the \emph{active} wavelets, i.e., those that are chosen by the adaptive scheme \textbf{SOLVE}. We see strong local refinements depending on the choice of the parameter, so that an adaptive discretization is obviously useful in this example.
\begin{figure}[!htb]
    \centering
    \begin{subfigure}{0.32\textwidth}
        \includegraphics[width=\textwidth]{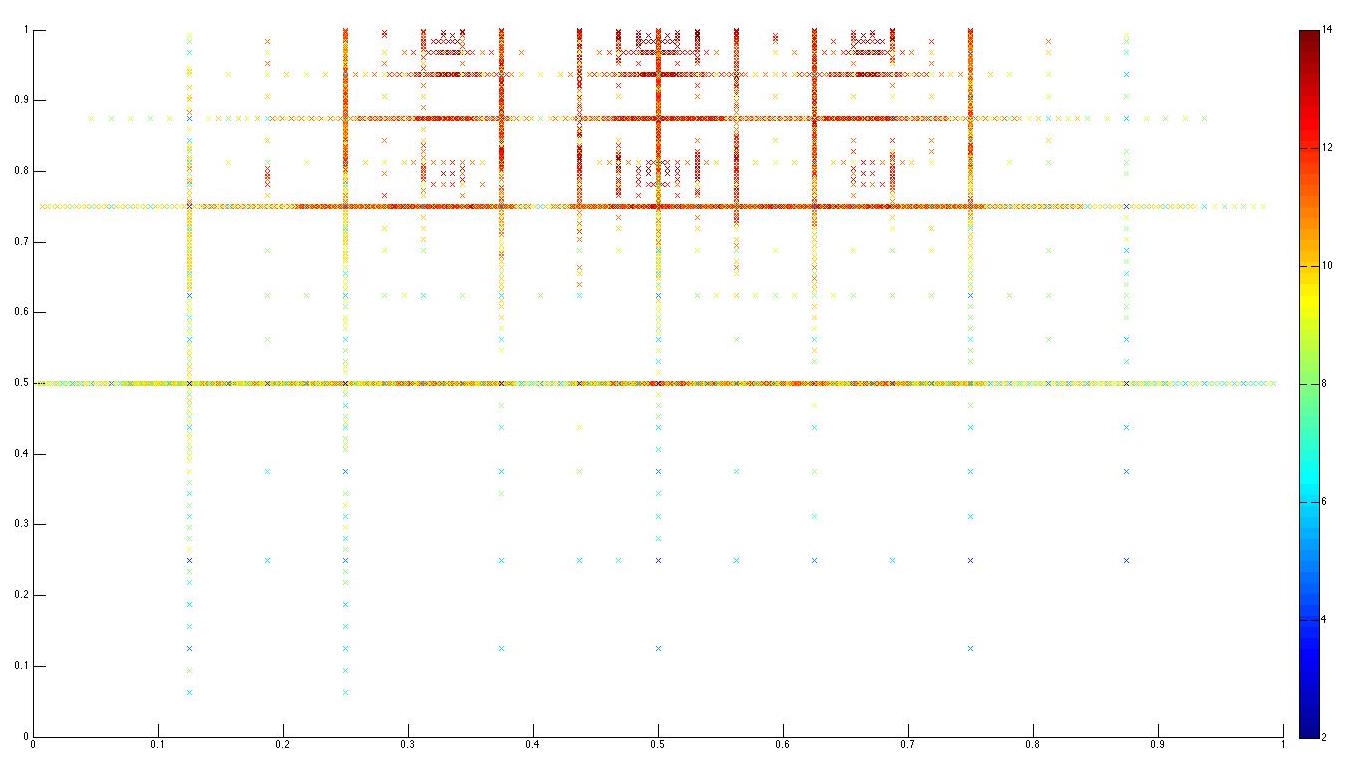}
        \caption{$\mu_1=10$, $\mu_2=6$}
        \label{fig:sup_mu6}
    \end{subfigure}
    \begin{subfigure}{0.32\textwidth}
        \includegraphics[width=\textwidth]{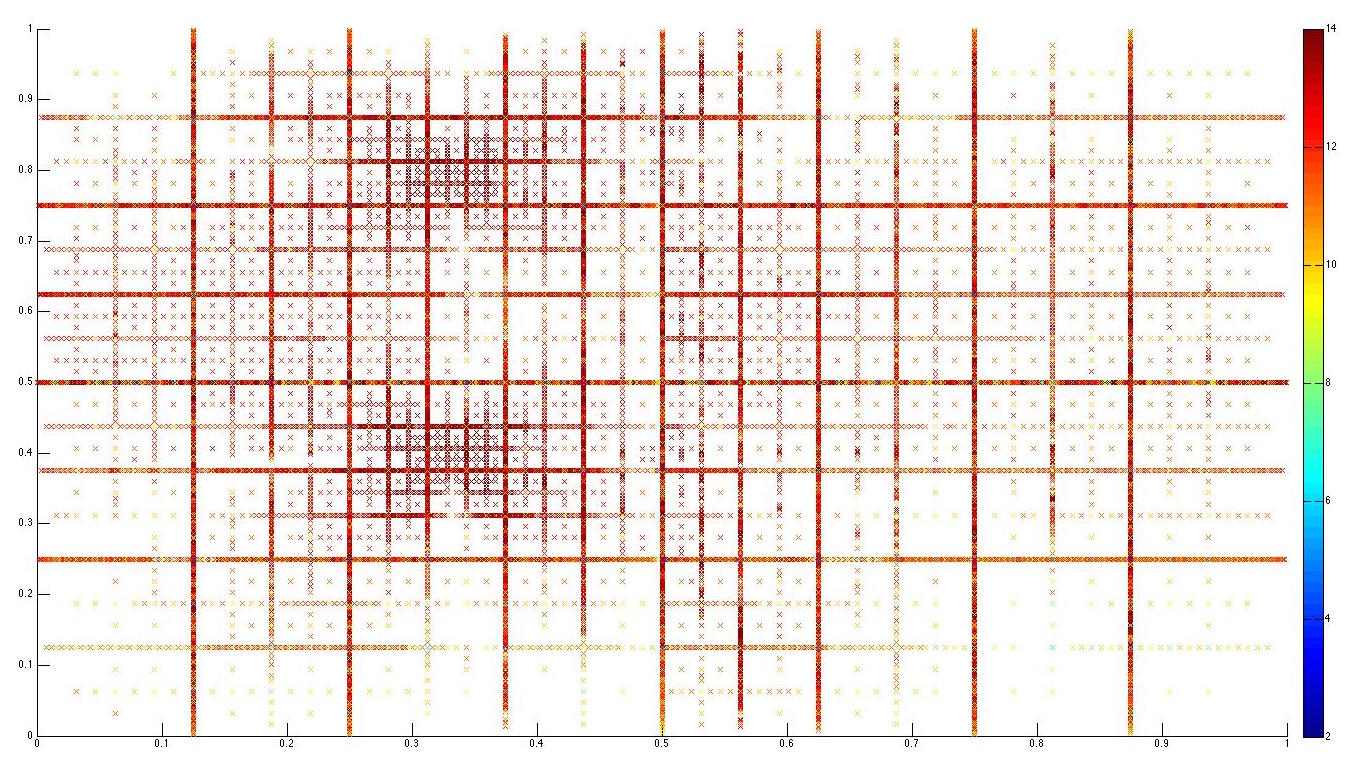}
        \caption{$\mu_1=0.0269$, $\mu_2=2$}
    \end{subfigure}
    \begin{subfigure}{0.32\textwidth}
        \includegraphics[width=\textwidth]{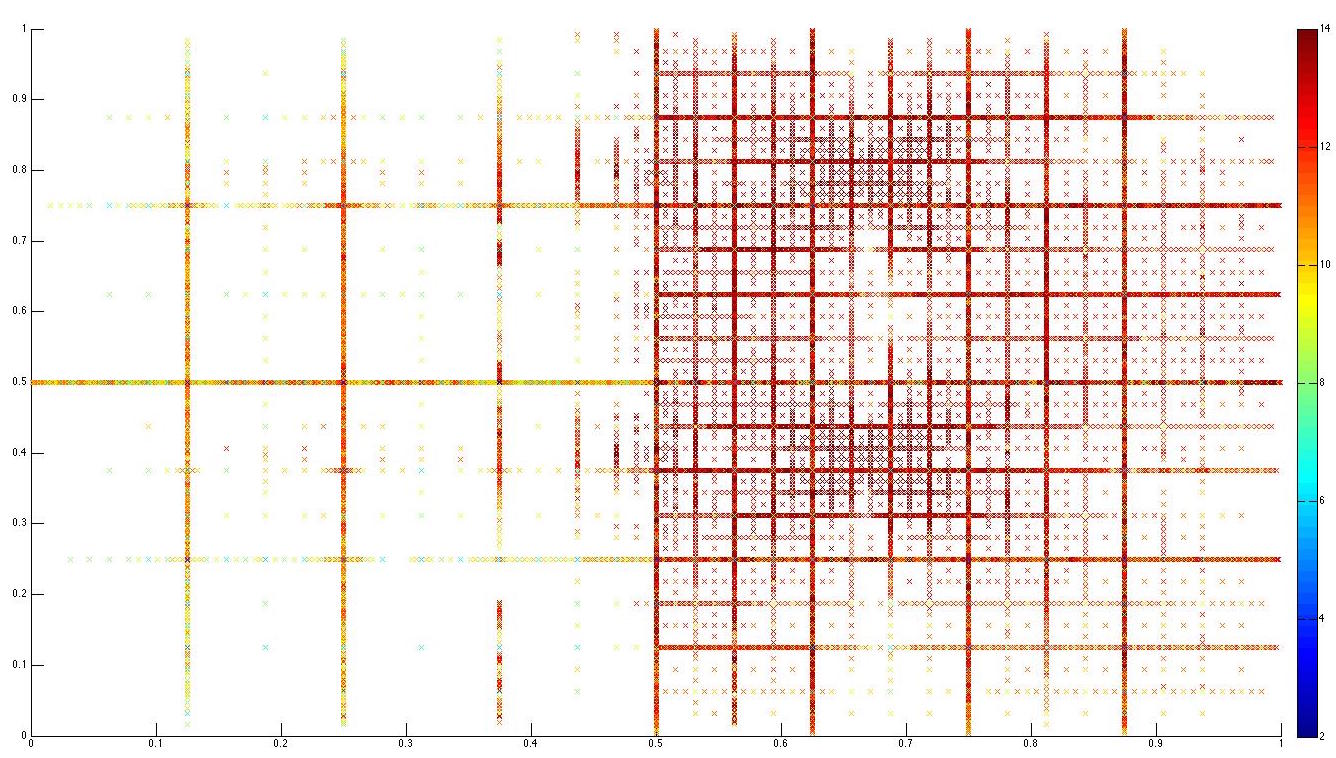}
        \caption{$\mu_1=0.01$, $\mu_2=8$}
        \label{fig:sup_mu8}
    \end{subfigure}
    \caption{Support centers of local wavelets of snapshots
             of the elliptic thermal block problem.
             The color bar indicates the size of the corresponding wavelet coefficients.}
    \label{fig:support}
\end{figure}

\subsubsection{Greedy Performance}
The training set $\cD_\text{train}$ consists of $20$ log-spaced values in $\cD_1=[0.01, 10]$ for $\mu_1$ and all $9$ possible values for $\mu_2$. We set the Greedy tolerance $\widetilde{\text{tol}} := 10^{-4}$. Table \ref{Tab:Seal} shows the decay of the error estimator w.r.t.\ to the number of basis functions. To achieve the specified tolerance, $24$ snapshots are required. We also indicate the chosen snapshots in Table \ref{Tab:Seal}.

\begin{figure}[!htb]
	\begin{minipage}{0.48\textwidth}
    	\begin{subfigure}{0.98\textwidth}
	    	\begin{center}
	        \footnotesize
			\begin{tabular}{|c|r|r|}\hline
			    $N$ & Snapshot & $\frac{\| \br_a(\bu_N^\eps(\mu^i); \mu^i)\|_
			                     {\ell_2}}{\| \br_a(\bu^\eps(\mu^i); \mu^i)\|_{\ell_2}}$
			                     \\ \hline 
			    1   & $(0.01, 2)$ & 1.0137\\
			    2   & $(0.01, 8)$ & 1.0392\\
			    3   & $(0.01, 5)$ & 1.0256\\
			    4   & $(0.01, 7)$ & 1.0317\\
			    5   & $(0.01, 1)$ & 1.0177\\
			    6   & $(0.01, 4)$ & 1.0342\\
			    7   & $(0.01, 9)$ & 1.0326\\
			    8   & $(0.01, 3)$ & 1.0119\\
			    9   & $(0.01, 6)$ & 1.0298\\
			    10  & $(0.02069, 5)$ & 1.0395\\
			    11  & $(0.02069, 4)$ & 1.0365\\
			    12  & $(0.02069, 6)$ & 1.0307\\
			    13  & $(0.02069, 2)$ & 1.4497\\
			    14  & $(0.02069, 1)$ & 1.3328\\
			    15  & $(0.02069, 3)$ & 1.2525\\
			    16  & $(10, 8)$ & 1.1594\\
			    17  & $(10, 4)$ & 1.0884\\
			    18  & $(10, 7)$ & 1.2069\\
			    19  & $(10, 5)$ & 1.0638\\
			    20  & $(10, 6)$ & 1.0917\\
			    21  & $(10, 9)$ & 1.2531\\
			    22  & $(0.02069, 7)$ & 1.0601\\
			    23  & $(0.02069, 8)$ & 1.0977\\
			    24  & $(0.02069, 9)$ & 1.0539\\\hline
			\end{tabular}
       		\caption{\footnotesize Selected snapshots and
                 residual deterioration rate.}
        			\label{Tab:Seal}
	    		\end{center}
		\end{subfigure}
	\end{minipage}
	\begin{minipage}{0.48\textwidth}
    	\centering
    	\begin{subfigure}{0.98\textwidth}
        \includegraphics[width=\textwidth]{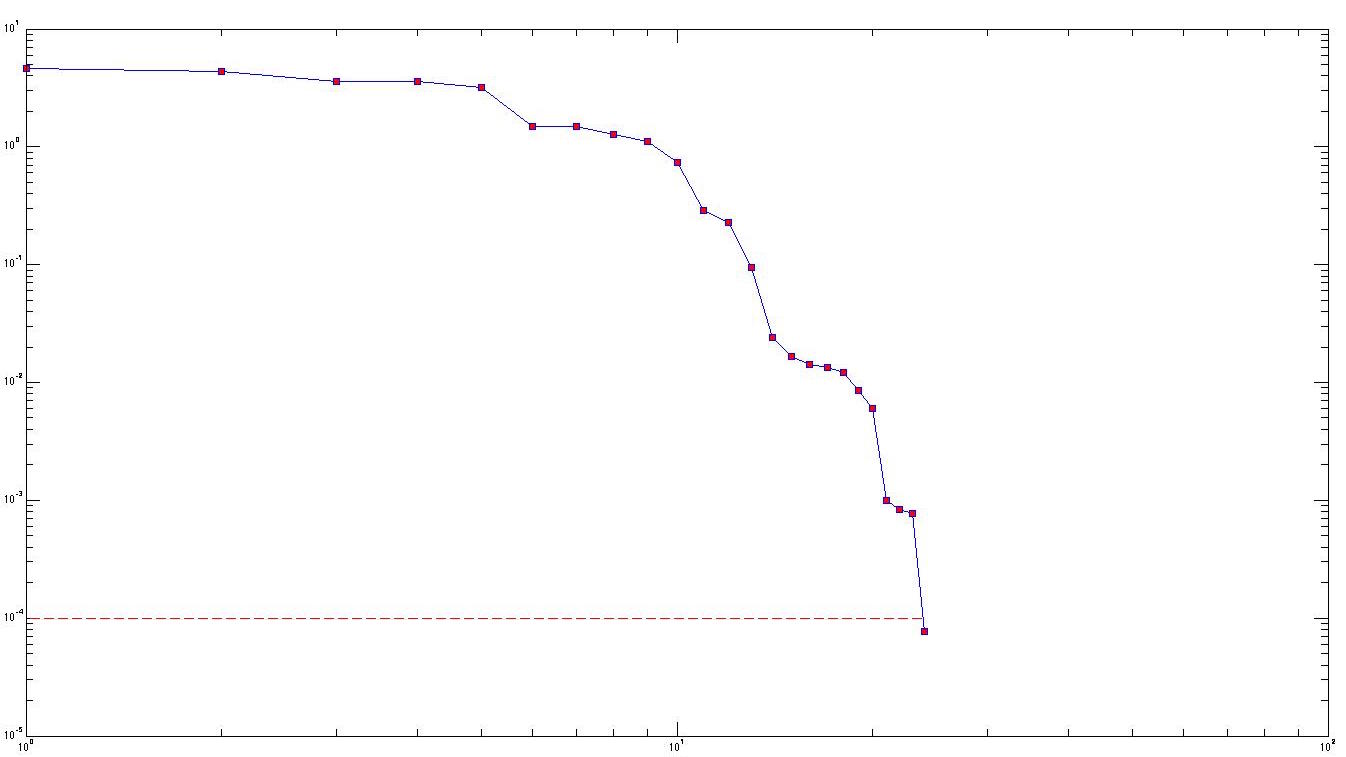}
        \caption{\footnotesize Weak Greedy training starting at $N=0$.}
        \label{fig:greedy_train}
    \end{subfigure}\newline
    \begin{subfigure}{0.98\textwidth}
        \includegraphics[width=\textwidth]{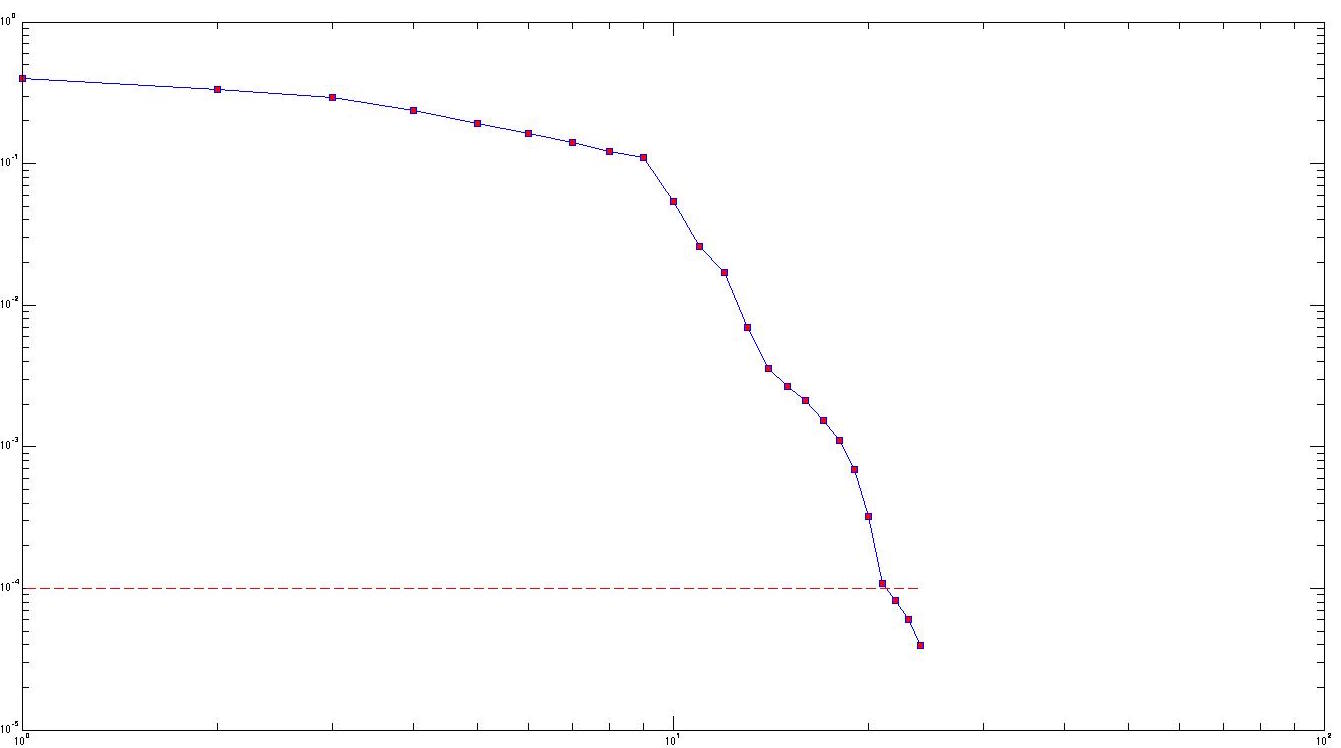}
        \caption{\footnotesize Average error estimator for 450 test parameters
             with $N=1,\ldots,24$ basis functions.}
        \label{fig:test_avgs}
    \end{subfigure}
	\end{minipage}
       \caption{Weak Greedy training for the elliptic thermal block problem.
       	 \label{fig:greedy_train-full}}
\end{figure}

As we see, due to the parameter-dependent locality of the source $f$ (and thus the highly localized solutions), adding a snapshot to the RB-basis for a certain $\mu_1$ does not necessarily add additional information for other parameters. Hence, the Greedy loop iterates through the $9$ values for $\mu_2$ first, i.e., all $9$ possible source locations. For $\mu_1=10^{-2}$, the coercivity constant $\alpha(\mu)$ is quite small and thus has a strong impact in the error estimator. As we see in Figure \ref{fig:greedy_train}, the convergence rate slightly stalls for those samples with the same $\mu_1$-value (i.e, same conductivity but different location of the source), but this effect becomes less and less pronounced for increasing $N$.

In Figure \ref{fig:test_avgs}, we show the average error estimator over a test set, which is chosen as 50 log-spaced values for $\mu_1$ in $[0.01, 20]$ and all 9 possible values for $\mu_2$ ($\#\cD_{\text{test}}=450$). We observe algebraic convergence for the first $9$ samples (corresponding to the different sources) and exponential rate of convergence afterwards. For $N=24$, we obtain a maximal error estimator over the test space of $1.00766\times 10^{-4}$ (which only slightly exceeds $\widetilde{\text{tol}} := 10^{-4}$). This value is attained at $\mu=(20, 6)$, which is outside the range of the training set $\cD_{\text{train}}$.

Let us now comment on the values in the third column of Table \ref{Tab:Seal}, where we indicate the ratio of the residual of the RB-approximation $u_N^\eps(\mu^i)$ and the snapshot $\zeta_i^\eps= u^\eps(\mu^i)$ for the chosen samples $\mu^i\in S_N$. Both residuals are the discrete ones, i.e., the vectors of the wavelet coefficients. These ratios are interesting for different reasons:
\begin{compactenum}
	\item They indicate the size of the reproduction error. Recall, that we cannot expect exact reproduction of snapshots. As the very moderate numbers indicate, RB and snapshot errors are of the same size which is a quite positive result.
	\item We have discussed in Section \ref{Sec:ErrorMeasure} necessary snapshot accuracies to ensure Greedy convergence, in particular the $\mu$-dependence of these accuracies. For this first example, $\cA(\mu)$ is elliptic and we determined $u_N^\eps(\mu)$ as the usual Galerkin solution in $X_N^\eps$ (without computing an approximation of the infinite-dimensional normal equation operator). Hence, from Proposition \ref{prop:a4}, we could expect a factor of $\sqrt{\gamma(\mu) / \alpha(\mu)}$, which is $10$ for $\mu_1=0.01$. As we see, we obtain numbers in the order of $1$, which means that the standard Galerkin RB-solution is very close to the optimal Galerkin solution of the infinite-dimensional normal equations. A possible explanation is that the wavelet preconditioning yields a very small spectrum and quantitatively good condition numbers.
	\item The adaptively computed snapshot $u^\eps(\mu^i)$ can also be interpreted as the full approximation of the RB-approximation $u_N^\eps(\mu^i)$ for the same sample value $\mu^i$. The snapshot, in turn, is guaranteed to be within the prescribed tolerance of the exact solution $u(\mu^i)$. This means that the numbers in the third column of Table \ref{Tab:Seal} are the effectivities of the adaptive wavelet-RB error estimator for the sample values $\mu^i$.
\end{compactenum}

\subsection{Parabolic Periodic Space-Time Equation}\label{Sec:Ex2}
\subsubsection{Data}
As a second example for our numerical experiments, we consider the time-periodic convection-diffusion-reaction (CDR) equation 
\begin{equation*}
	\left\{
	\begin{aligned}
	u_{t} - u_{xx} + \mu_{1} \beta (x) u_{x} + \mu_{2} u &= \cos(2\pi t) \qquad\, \text{ on } \Omega = (0,1),\\
			u(t,0) &= u(t,1) \hspace{1.1cm} \text{ for all } t \in [0,T],\\
			u(0,x) = u(T,x)&= 0 \hspace{2.05cm} \!\text{on } \overline\Omega,
	\end{aligned}\right.
\end{equation*}
with coefficient function $\beta (x)= 0.5-x$. 
Setting $V:=H^1_0(\Omega)$, $H^1_\per(0,T) := \{ v \in H^1(0,T): v(0) = v(T) \}$, we define the spaces $\cY := L_2(0,T; V )$ and $\cX := L_2 (0,T; V ) \cap H^1_{\per}(0,T; V')$, i.e.,
\begin{align}
   \cX &= \{ v \in L_2 (0,T; V ): v_{t} \in L_2 (0,T;V'),\ v(0) = v(T) \text{ in } H\}, \label{eq:cX} 
\end{align}
where $\cX$ is equipped with the norm $\norm{v}^2_{\cX} := \norm{v}_{L_{2}(0,T;V)}^{2} + \norm{v_{t}}_{L_{2}(0,T;V')}^{2}$, $v \in \cX$. Note that $v(0)$, $v(T)$ are well-defined due to $H^{1}(0,T) \subset {C([0,T])}$ and $\{ v \in L_2 (0,T; V ): v_{t} \in L_2 (0,T;V')\} \subset C(0,T;H)$, e.g., \cite{DautrayLions}.  We obtain the variational problem:
\begin{align}\label{eq:problem_weakform}
	 \text{Find } u \in \cX: \qquad b(u,v; \mu) = f(v) \qquad \forall\, v \in \cY,\, \mu=(\mu_1, \mu_2),
\end{align}
with forms $b(\cdot, \cdot; \mu): \cX \times \cY \times \cD \to \R$, $f(\cdot): \cY \to \R$ given by
\begin{align}
  b(u,v; \mu) &:= \int_0^T [\eval{v(t)}{u_{t}(t)}_{V\times V'} + a(u(t),v(t); \mu)] dt, \,
  	\label{eq:bil_form_b}
\end{align}
where $f(v) := \int_0^T \cos(2\pi t) \eval{v(t)}{1}_{V\times V'}  dt$ and the bilinear form as $a(\phi, \eta; \mu) = (\phi_x, \eta_x)_{L_2(\Omega)} + \mu_1 (\beta \phi_x, \eta)_{L_2(\Omega)} + \mu_2 (\phi, \eta)_{L_2(\Omega)}$.

\subsubsection{Discretization}
As  bases we use space-time tensor functions: in time, we use a collection of bi-orthogonal B-spline wavelets on $\R$ of order $d_{t} =  m_{t}=2$, periodized onto $[0,T]$, \cite{Urban:WaveletBook}. The spatial basis is chosen as bi-orthogonal B-spline wavelets of order $d_{x}=m_{x}=2$ with homogeneous boundary conditions from \cite{Dijkema:Diss}. The test basis is  a tensor product of the above mentioned linear B-spline wavelets with $2$ vanishing moments from \cite{Dijkema:Diss} with homogenous boundary conditions in the univariate spatial basis.

In this example, the snapshots have different temporal evolutions. Since time is a `normal' variable in a space-time variational formulation, this means that different discretizations for the snapshots in space-time may pay off. In particular, the right-hand side is smooth, hence we do not expect strong local effects as in the previous example. Different snapshots merely exhibit different temporal evolutions as can be seen in Figure \ref{fig:cdr_sols}. This justifies adaptivity.
\begin{figure}[!htb]
	    \includegraphics[width=\textwidth]{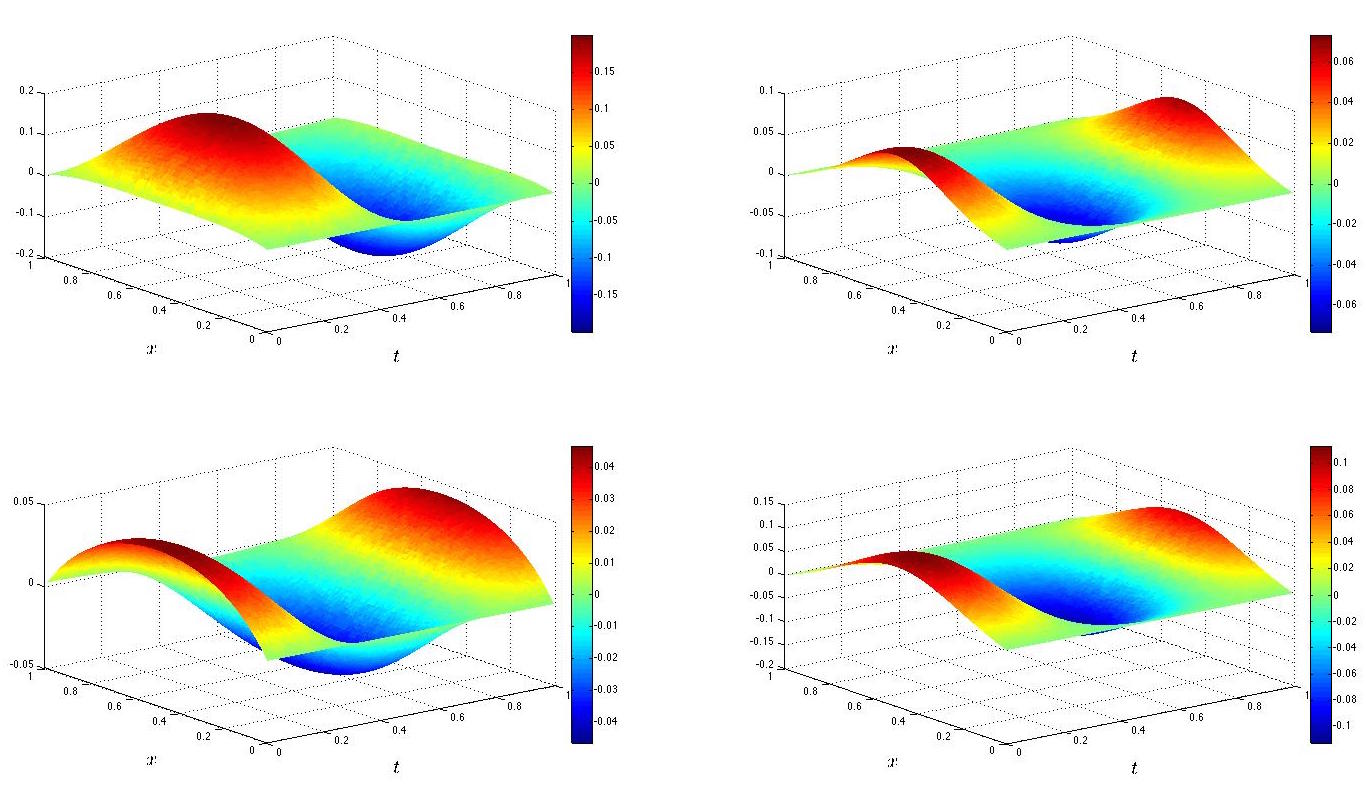}
    		\caption{Snapshots for the CDR problem for
             	$\mu=(0,-9),(30,-9)$ in the first row
             	and $\mu=(0,15),(17.3684,-9)$ in the second.}
    			\label{fig:cdr_sols}
\end{figure}

\subsubsection{Greedy Performance}
We perform the same experiments as in the elliptic thermal block problem for the CDR equation. The training set was chosen as 20 uniformly spaced values in $[0,30]$ for $\mu_1$ and in $[-9,15]$ for $\mu_2$. The Greedy tolerance was set to $\widetilde{\text{tol}}:=10^{-4}$. 

The decay of the error estimator is shown in Figure \ref{fig:greedy_train_CDR}, where we also see that $8$ basis functions suffice to reach the desired tolerance. The chosen snapshots are listed in Table \ref{tab:snapshots_CDR}. The different temporal evolution of snapshots is reflected by the convergence history in Figures \ref{fig:greedy_train_CDR} and \ref{fig:test_avgs_CDR}. In the latter figure, we test the error estimator again on a larger test set, which is here chosen within the same range as the training set, but with a finer uniform discretization of 50 values for $\mu_1$ and $\mu_2$. The maximal error stays below the Greedy tolerance $\widetilde{\text{tol}}=10^{-4}$. We obtain a maximal error of $6.1410\times 10^{-5}$ on the training set and of $6.2397\times 10^{-5}$ on the larger test set, i.e., just a very mild increase. However, this is not surprising at all, since the training set covers the same range as the test set.

As in the elliptic case of the thermal seal-block, the numbers in the third column of Table \ref{tab:snapshots_CDR} are quite positive. In this case, there are also additional features:
\begin{compactenum}
	\item Here, we face a Petrov-Galerkin problem. We have proven above that $Y_N^\eps(\mu):=\cR_\cY'\, \cB(\mu)(X_N^\eps)$ is an optimal test space. In our numerical experiments, however, we used $Y_N^\eps(\mu)\equiv X_N^\eps$ for simplicity. This is possible in this case since $\cX\subset\cY$. The values in Table \ref{tab:snapshots_CDR} show that this simple choice gives quite good results.
	\item At a first glance, the numbers less than $1$ seem surprising. However, recall that the right-hand side is parameter-\emph{independent}. These means that the RB-space $X_N^\eps$ may contain components that improve the RB-approximation over the snapshot, yielding ratios smaller than $1$.
	\item Finally, we obtain an RB-system with only $8$ basis functions. In the online stage, we thus only have to solve one $8\times 8$ linear system for the full evolution -- no temporal iteration (time-stepping) is required. This gives rise to an enormous speedup which allows us to perform 2500 tests on a PC in about 2 seconds!
\end{compactenum}
\begin{figure}[!htb]
	\begin{minipage}{0.48\textwidth}
    \begin{subfigure}{0.98\textwidth}
    	\begin{center}
        \footnotesize
		\begin{tabular}{|c|r|r|}\hline
		    $N$ 
		    & Snapshot 
		    & $\frac{\|\br(\bu_N^\varepsilon(\mu^i);\mu^i)\|_{\ell_2}}{\|\br(\bu^\varepsilon(\mu^i); \mu^i)\|_{\ell_2}}$\\ \hline
		    1   & $(0, -9)$ 					& 1.0065\\
		    2   & $(1.5789, -9)$ 			& 0.9990\\
		    3   & $(0, -7.7368)$ 			& 0.9994\\
		    4   & $(30, -9)$ 				& 1.0183\\
		    5   & $(9.4737, -9)$ 			& 0.9967\\
		    6   & $(0, 15)$ 					& 1.0255\\
		    7   & $(20.5263, -2.6842)$ 	& 1.0162\\
		    8   & $(17.3684, -9)$ 			& 0.9942\\ \hline
		\end{tabular}
    	\end{center}
	\caption{\label{tab:snapshots_CDR}Selected snapshots and
                 residual deterioration rate. 
		}
	\end{subfigure}
	\end{minipage}
	\begin{minipage}{0.48\textwidth}
    	\centering
    \begin{subfigure}{0.98\textwidth}
        \includegraphics[width=\textwidth]{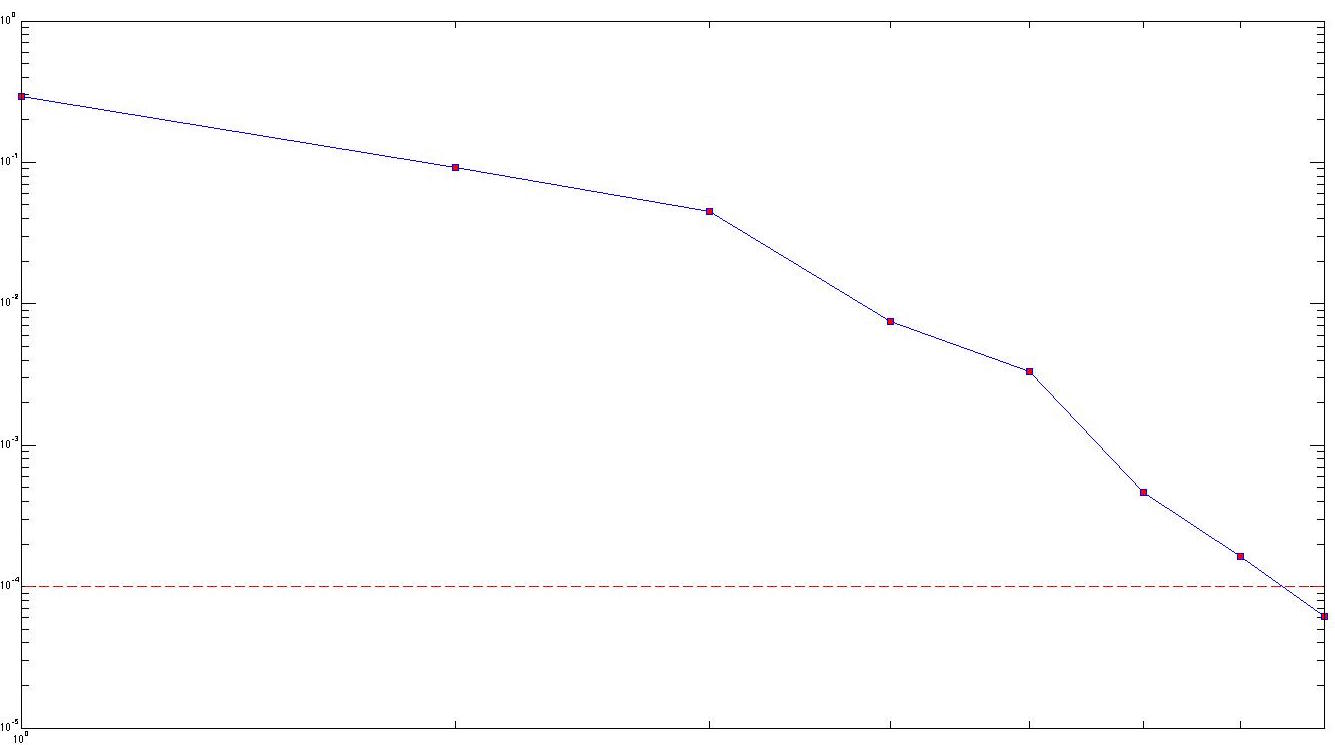}
        \caption{\footnotesize Weak Greedy training starting at $N=0$.}
        \label{fig:greedy_train_CDR}
    \end{subfigure}\newline
    \begin{subfigure}{0.98\textwidth}
        \includegraphics[width=\textwidth]{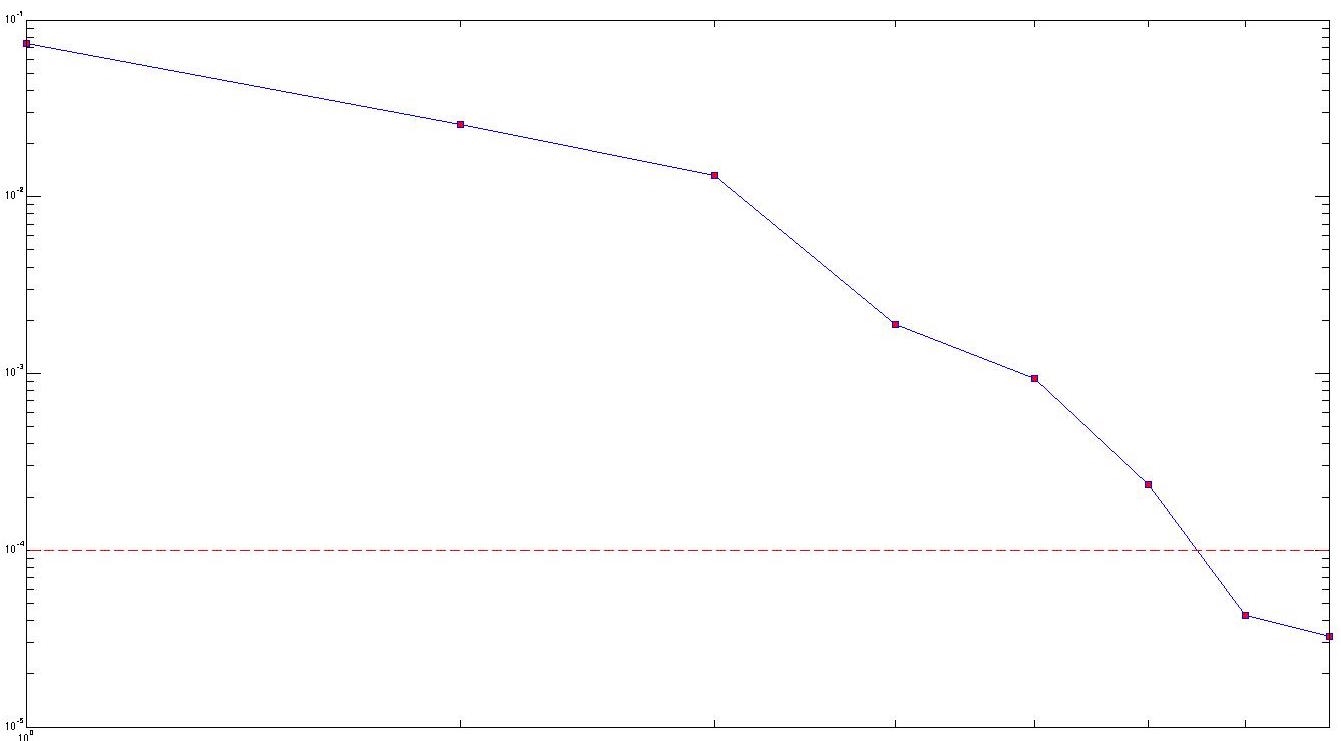}
        \caption{\footnotesize Average error estimator for 2500 test parameters
             with $N=1,\ldots,8$ basis functions.}
        \label{fig:test_avgs_CDR}
    \end{subfigure}
	\end{minipage}
       \caption{Greedy performance for the CDR problem.
       	 \label{fig:cdr}}
\end{figure}


\begin{appendix}

\section{A sharp residual estimate using normal equations}\label{App:A}

We consider Hilbert spaces $X$, $Y$ with their duals $X'$, $Y'$ induced by some pivot spaces. Let $B:X\to Y'$ be a linear operator that satisfies the Ne\v{c}as conditions, in particular $B$ is assumed to be continuous with constant $\gamma<\infty$ and to satisfy an inf-sup-condition with constant $\beta>0$. Thus, in particular, the operator equation $Bu=f$ has a unique solution $u\in X$ for any $f\in Y'$, i.e., $b(u,v)=\langle f,v\rangle_{Y'\times Y}$ for all $v\in Y$, where $\langle Bu,v\rangle_{Y'\times Y} =: b(u,v)$ for $u\in X$ and $v\in Y$.

Next, we consider the Petrov-Galerkin method with finite-dimensional spaces $X_h\subset X$, $Y_h\subset Y$ and
\begin{equation}\label{a:0}
	u_h\in X_h:\qquad 
	b(u_h, v_h) = \langle f, v_h\rangle_{Y'\times Y}\quad \forall v_h\in Y_h.
\end{equation}
We assume well-posedness also of the discrete problem independent of the discretization parameter $h$. 
The following C\'{e}a-type lemma is well-known
\begin{equation}\label{a:1}
	\| u-u_h\|_X \le \frac{\gamma}{\beta} \inf_{w_h\in X_h} \| u-w_h\|_X,
\end{equation}
see, e.g., \cite[Thm.\ 2]{MR1971217}. If $A: X\to X'$ is selfadjoint and positive definite with coercivity constant $\alpha>0$ (i.e., the elliptic case), this can be improved as follows
\begin{equation}\label{a:2}
	\| u-u_h\|_X \le {\frac{\sqrt\gamma}{\sqrt\alpha}} \inf_{w_h\in X_h} \| u-w_h\|_X.
\end{equation}
We associate a bilinear form to the operator $A$ as usual by $a(u,w):=\langle Au, w\rangle_{X'\times X}$ for $u,w\in X$.

Both estimates \eqref{a:1} and \eqref{a:2} relate the (Petrov-)Galerkin error to the error of the best approximation in $X_h$ -- both measured in an appropriate norm, here $\|\cdot\|_X$. Since the error is usually not available (the exact solution is unknown), but the residual is in some cases (e.g., both in RB `truth' discretizations and in adaptive wavelet methods), we want to relate the (Petrov-)Galerkin residual
$$
	r_b(u_h) := f- B u_h
	\qquad\text{or}\qquad
	\| r_b(u_h)\|_{Y'}
$$
to the best approximation residual $\inf\limits_{w_h\in X_h} \| r_b(w_h)\|_{Y'}$. The following estimate is straightforward by \eqref{a:1}
\begin{align}
	\| r_b(u_h)\|_{Y'}
	&= \| f - Bu_h\|_{Y'}
	=\kern-1pt \| B(u-u_h)\|_{Y'}  
	\kern-2pt\le\kern-1pt \gamma \| u-u_h\|_{X} 
	\le \frac{\gamma^2}{\beta}\kern-1pt \inf_{w_h\in X_h} \| u-w_h\|_X 
	\nonumber\\
	&
	\le \frac{\gamma^2}{\beta^2} \inf_{w_h\in X_h} \| B(u-w_h)\|_{Y'} 
	= \Big(\frac{\gamma}{\beta}\Big)^2 \inf_{w_h\in X_h} \| r_b(w_h)\|_{Y'}, 
	\label{a:3}
\end{align}
but the factor $\big(\frac{\gamma}{\beta}\big)^2$ is not satisfactory since it may be large. 

In the elliptic case, we can use the fact the Galerkin projection is the orthogonal projection of $u$ with respect to the energy scalar product $a(\cdot,\cdot)$ that induces the energy norm, i.e.,
$$
	\| w\|_a^2 := a(w,w),\,\, w\in X,
	\qquad
	\| f\|_{a'} := \sup_{w\in X} \frac{\langle f, w\rangle_{X'\times X}}{\| w\|_a},\,\, f\in X',
$$
(i.e., $\sqrt\alpha\, \| w\|_{X} \le \| w\|_{a} \le \sqrt\gamma\, \| w\|_{X}$ and $\frac1{\sqrt\gamma}\, \| f\|_{X'} \le \| f\|_{a'} \le \frac1{\sqrt\alpha}\, \| f\|_{X'}$), so that
\begin{align}
	\| r_a(u_h)\|_{X'}
	&\le \frac1{\sqrt\alpha} \| r_a(u_h)\|_{a'}
	= \frac1{\sqrt\alpha} \inf_{w_h\in X_h} \| r_a(w_h)\|_{a'} \nonumber \\
	&\le \frac{\sqrt\gamma}{\sqrt\alpha} \inf_{w_h\in X_h} \| r_a(w_h)\|_{X'}. 
	\label{a:4}
\end{align}

If we compare \eqref{a:2} and \eqref{a:4} -- both in the elliptic case -- we see that the factor $\frac{\sqrt\gamma}{\sqrt\alpha}$ is the same, both in the error relation in \eqref{a:2} and in the residual relation \eqref{a:4}. In the general case, however, we have a factor of $\frac{\gamma}{\beta}$ for the error relation in \eqref{a:1}, but the square $\big(\frac{\gamma}{\beta}\big)^2$ in the residual relation \eqref{a:4}.

The aim of this appendix is to improve and harmonize these estimates for the residual relations. To this end, we introduce some operators.
\begin{definition}\label{def:A1}
(a) The operator $B^+: Y'\to X'$ is defined as $\langle B^+ v', u\rangle_{X'\times X} := (Bu, v')_{Y'}$ for $v'\in Y'$ and all $u\in X$. \\
(b) The \emph{Riesz operator} $R_Y:Y'\to Y$ is defined as $\langle R_Y y', z'\rangle_{Y\times Y'} := (y',z')_{Y'}$ for all $y', z'\in Y'$. \\
(c) The \emph{dual operator} (Banach adjoint) $C':Y'\to X'$ of a linear operator $C:X\to Y$ is defined as $\langle C'y', x\rangle_{X'\times X} = \langle Cx, y'\rangle_{Y\times Y'}$, $x\in X$, $y\in Y$.
\end{definition}
We note in particular, that the operator $R_Y$ is invertible and coincides with the Riesz representation operator. Moreover, $R_Y';Y'\to Y$ is given by $\langle R_Y'y',z'\rangle_{Y\times Y'} = \langle R_Y z', y'\rangle_{Y\times Y'} = (y',z')_{Y'}$ for $y',z'\in Y'$ as well as $B':Y\to X'$ as $\langle B'y,x\rangle_{X'\times X} = \langle y, Bx\rangle_{Y\times Y'}$ for $x\in X$, $y\in Y$.

\begin{lemma}\label{La:a1}
	It holds that $B^+ = B' R_Y$.
\end{lemma}
\begin{proof}
	Let $v'\in Y'$ and $u\in X$. Then, the following is easily seen, 
	$\langle B^+ v', u\rangle_{X'\times X}
		= (Bu, v')_{Y'}
		= \langle Bu, R_Y v'\rangle_{Y'\times Y}
		= \langle B' R_Y v', u\rangle_{X'\times X}$, 
	which proves the claim.
\end{proof}
\smallskip

As an immediate consequence, we get $\| B^+\|_{\cL(Y',X')} = \| B\|_{\cL(X,Y')}$, i.e., the estimate $\beta \| v'\|_{Y'} \le \| B^+ v'\|_{Y'} \le \gamma \| v'\|_{Y'}$ for all $v'\in Y'$.

\begin{proposition}\label{prop:a3}
Let $B: X\to Y'$ be linear and invertible. Then,

(a) $A:=B^+B: X\to X'$ is elliptic w.r.t.\ the norm $\|B\cdot\|_{Y'}$ with $\alpha=\gamma=1$.

(b) $Bu=f$ in $Y'$ if and only if $Au=B^+f=:g$ in $X'$. 
\end{proposition}
\begin{proof}
Let us first show that $A$ is selfadjoint. In fact, for $u,w\in X$, we have
\begin{align}
	\langle Au,w\rangle_{X'\times X}
	&= \langle B^+B u,w\rangle_{X'\times X}
	= (Bu, Bw)_{Y'},
\end{align}
	which allows to exchange $u$ and $w$. This implies that
\begin{align}
	\langle Au,u\rangle_{X'\times X} = \| Bu\|_{Y'}^2,
		\qquad
	\langle Au,w\rangle_{X'\times X} \le \| Bu\|_{Y'}\, \| Bw\|_{Y'},
\end{align}
which proves (a). In order to prove (b), let $w\in X$ arbitrary. Then,
\begin{align*}
 	\langle g, w\rangle_{X'\times X} 
	&= \langle B^+f , w\rangle_{X'\times X}
	= \langle B' R_Y f , w\rangle_{X'\times X}
	= \langle R_Y f , B w\rangle_{Y\times Y'} 
	\\ 
	&= \langle  f , R'_Y B w\rangle_{Y'\times Y}  
\end{align*}
and on the other hand,
\begin{align*}
	\langle Au, w\rangle_{X'\times X}
	&= \langle B^+B u,w\rangle_{X'\times X}
	= \langle B' R_Y B u,w\rangle_{X'\times X}
	= \langle R_Y B u, Bw\rangle_{Y\times Y'}\\
	&= \langle  B u, R'_Y Bw\rangle_{Y'\times Y}.
\end{align*}
Since both $R_Y$ and $B$ are invertible, also the operator $R_Y'B= (B^+)': X\to Y$ also invertible, so that the assertion is proven.
\end{proof}

Let us now consider the Galerkin problem with respect to the elliptic operator $A$ and a finite-dimensional space $X_h\subset X$, i.e., 
\begin{equation}\label{Eq:Gal}
	\bar u_h\in X_h:\quad
	a(\bar u_h, w_h) 
	= \langle A\bar u_h, w_h\rangle_{X'\times X} 
	= \langle g, w_h\rangle_{X'\times X}
	\quad
	\forall w_h\in X_h.
\end{equation}
This is equivalent to the normal equations $(B u_h, Bw_h)_{Y'}= (f, Bw_h)_{Y'}$ for all $w_h\in X_h$.  It should be noted that \eqref{Eq:Gal} is \emph{not} the normal equation of the discrete Petrov-Galerkin problem \eqref{a:0}. In fact, \eqref{a:0} is equivalent to $B_hu_h=f_h$ in $Y_h=Y_h'$ (finite dimensional), where $B_h:= B_{|Y_h\times X_h}$, $f_h:=f_{|Y_h}$. Thus, the (discrete) normal equation reads $A_h u_h = B_h^Tf_h$ with $A_h:=B_h^TB_h$. On the other hand, \eqref{Eq:Gal} reads $\bar A_h \bar u_h = g_h$ in $X_h$, where $\bar A_h := (B^+B)_{|X_h\times X_h}$ and $g_h:=(B^+f)_{|X_h}$. This means that $B^+B$ is first multiplied exactly (on the operator level) and then discretized, whereas $A_h$ is first discretized and then multiplied on the discrete level. 

\begin{proposition}\label{prop:a4}
	Under the above assumptions, we have
	\begin{equation}\label{eq:bop}
		\| r_b(\bar u_h) \|_{Y'} 
		= \| B(u-\bar u_h)\|_{Y'} 
		= \inf_{w_h\in X_h} \| B(u-w_h)\|_{Y'}
		= \inf_{w_h\in X_h} \| r_b(w_h)\|_{Y'}
	\end{equation}
	as well as
	\begin{equation}\label{eq:bopa}
		\| r_a(\bar u_h) \|_{X'} 
		\le \frac{\gamma}{\beta} \inf_{w_h\in X_h} \| r_a(w_h)\|_{X'}.
	\end{equation}
	
	If $B:X\to X$ is elliptic, we get $\| r_b(\bar u_h) \|_{X'} 
		\le \frac{\sqrt\gamma}{\sqrt\alpha} \inf_{w_h\in X_h} \| r_b(w_h)\|_{X'}$ as well as $\| r_a(\bar u_h) \|_{X'} 
		\le \frac{\gamma}{\alpha} \inf_{w_h\in X_h} \| r_a(w_h)\|_{X'}$ for $A=B^+B$.
\end{proposition}
\begin{proof}
	The result in \eqref{eq:bop} follows from Galerkin orthogonality in \eqref{eq:a8} for any $w_h\in X_h$
	\begin{align}
		\| B(u-\bar u_h)\|_{Y'}^2
		&= (B(u-\bar u_h), B(u-\bar u_h))_{Y'}
		= \langle A(u-\bar u_h), u-\bar u_h \rangle_{X'\times X} 
			\nonumber \\
		&= a(u-\bar u_h, u-\bar u_h)
			\nonumber \\
		&= a(u-\bar u_h, u-w_h) = (B(u-\bar u_h), B(u-w_h))_{Y'}
			\label{eq:a8}	 \\
		&\le \| B(u-\bar u_h)\|_{Y'}\, \| B(u-w_h)\|_{Y'}, 
			\nonumber
	\end{align}
	which proves \eqref{eq:bop}. Then, we use \eqref{eq:bop} to obtain
		\begin{align*}
			\| r_a(\bar u_h)\|_{X'}
			&= \| B^+ r_b(\bar u_h)\|_{X'}
			\le \gamma\, \| r_b(\bar u_h)\|_{Y'} 
			= \gamma \inf_{w_h\in X_h} \| r_b(w_h)\|_{Y'} \\
			&\le \frac\gamma\alpha \inf_{w_h\in X_h} \| B^+ r_b(w_h)\|_{X'}
			\le \frac\gamma\alpha \inf_{w_h\in X_h} \| r_a(w_h)\|_{X'},
		\end{align*}
		so that \eqref{eq:bopa}  is proven. The estimate in the elliptic case is proven as \eqref{a:4}.
\end{proof}

\begin{remark}\label{rem:a5}
	Let us comment on the previous result. Obviously, \eqref{eq:bop} significantly improves \eqref{a:3} so that the `ellipticity gap' (i.e., the quotient of the factor in the respective relation for the elliptic and the general inf-sup case) between \eqref{eq:bopa} and \eqref{a:4} is the same as the gap between \eqref{a:2} and \eqref{a:1} for the error, namely the factor $\frac{\sqrt\gamma}{\sqrt\alpha}$.
	\end{remark}

We point out again that $\bar u_h\in X_h$ is the solution of $\bar A_h \bar u_h=g_h$ which is a discretization of the infinite-dimensional version of the normal equations. As described in Section \ref{sec:awgm}, this is exactly what is done in adaptive wavelet methods (since there an optimal preconditioning of large classes of operators $B$ is available which means that one avoids the usual drawback of squaring a possibly bad condition number of $B$ when using normal equations). 

In general, the discrete approximation $\bar u_h$ of the infinite-dimensional normal equation differs from the Petrov-Galerkin solution $u_h$ in \eqref{a:0} (which, in turn, coincides with the solution of the discrete normal equations). Only in a very specific situation (which we will indicate now), both approximations coincide.

\begin{proposition}\label{prop:A6}
	Let $X_h\in X$ and $Y_h\in Y$ be finite-dimensional spaces. Consider
	\begin{align}
		\bar u_h\in X_h: 
			&& a(\bar u_h, w_h) &= \langle B^+ f, w_h\rangle_{X'\times X}
			& \forall w_h\in X_h,  
			\label{eq:G}\\
		u_h\in X_h: 
			&& b(u_h, v_h) &= \langle f, v_h\rangle_{Y'\times Y}
			& \forall v_h\in Y_h.
			\label{eq:PG}
	\end{align}
	If $X_h$ and $Y_h$ are related by $Y_h=(B^+)'(X_h)$, then $\bar u_h=u_h$.
\end{proposition}
\begin{proof}
	We follow the lines of the proof of Proposition \ref{prop:a3} (b), i.e.,
	$a(u_h, w_h) = b(u_h, R_Y'B w_h) = b(u_h, (B^+)'w_h)$ on one hand and on the other 
	$\langle B^+ f, w_h\rangle_{X'\times X} = \langle f, (B^+)' w_h\rangle_{Y'\times Y} = \langle f, R_Y' B w_h\rangle_{Y'\times Y} $.
	
	Let $\bar u_h$ solve \eqref{eq:G}. Since $B^+$ is invertible and due to $Y_h=(B^+)'(X_h)$, for any $v_h\in Y_h$, there is a unique $w_h\in X_h$ such that $v_h=(B^+)' w_h$. Hence, $\bar u_h$ also solves \eqref{eq:PG}. 
	On the other hand, let $u_h$ solve \eqref{eq:PG}. For any $w_h\in X_h$, there is a unique $v_h\in Y_h$ such that $v_h=(B^+)' w_h$ which implies that $u_h$ is a solution of  \eqref{eq:G} as well. 
	Since both \eqref{eq:G} and \eqref{eq:PG} admit a unique solution, we get $\bar u_h=u_h$.
\end{proof}

\begin{remark}\label{rem:a7}
This latter result shows that the Petrov-Galerkin solution (which coincides with the discrete normal equation solution) is the same as the approximation of the infinite-dimensional normal equation solution provided that trial and test spaces are chosen appropriately. In that case, we obtain the optimal residual relation (for $r_b$ even with constant $1$) from Proposition \ref{prop:a4} -- otherwise we cannot hope for it.
\end{remark}

\end{appendix}

\bibliographystyle{plain}
\bibliography{literature}

\end{document}